\documentclass[11pt]{article}
\usepackage{amsthm}
\usepackage{amsfonts}
\usepackage[T1]{fontenc}
\usepackage{epsfig}
\usepackage{hyperref}
\usepackage{wrapfig} 
\usepackage[lflt]{floatflt}
\usepackage{epsfig}
\usepackage{graphicx}
\usepackage{amsmath}
\usepackage{amssymb}
\usepackage{color}
\usepackage{subfigure}
\usepackage[english]{babel}
\usepackage{empheq}
\usepackage{caption}
\usepackage{stmaryrd}

\usepackage{makeidx}
\usepackage{latexsym}
\usepackage{enumerate}
\newtheorem{note}{Note}
\newtheorem{theo}{Theorem}
\newtheorem{defi}{Definition}
\newtheorem{prop}{Proposition}
\newtheorem{lemma}{Lemma}
\newtheorem{coro}{Corollary}
\newtheorem{remark}{Remark}

\newcommand{\bbN}{{\mathbb N}}
\newcommand{\bbR}{{\mathbb R}}

\newcommand{\al}{\alpha}

\newcommand{\p}{\partial}

\newcommand{\G}{\Gamma}

\newcommand{\rmd}{{\rm d}}

\usepackage{anysize}
\marginsize{2cm}{2cm}{1.5cm}{1.5cm}

\newcommand{\eqq}[2]{\begin{equation}  #1  \label{#2} \end{equation}    }
\newcommand*{\norm}[1]{\left\Vert{#1}\right\Vert}
\newcommand*{\abs}[1]{\left\vert{#1}\right\vert}

\newcommand*{\ca}{c_{\alpha}}

\newcommand*{\izj}{\int_{0}^{1}}

\newcommand*{\poch}{\frac{\p}{\p x}}

\newcommand*{\uz}{u_{0}}

\newcommand*{\ia}{I^{\alpha}}

\newcommand*{\vf}{\varphi}
\newcommand*{\ve}{\varepsilon}

\newcommand{\hd}{\hspace{0.2cm}}
\newcommand{\no}{\noindent}
\newcommand{\m}[1]{\mbox{#1}}

\newcommand{\da}{D^{\alpha}_x}

\newcommand{\um}{u^{m}}
\newcommand{\sm}{s^{m}}

\newcommand{\izt}{\int_{0}^{t}}

\newcommand{\dd}{\mathcal{D}_{\al}}

\newcommand{\ld}{L^{2}(0,1)}

\begin{document}

\begin{center}\emph{}
\LARGE
\textbf{ A one-phase space - fractional Stefan problem with no  liquid initial domain}
\end{center}
 \normalsize
  \begin{center}
  Sabrina Roscani$^{2,3}$, Katarzyna Ryszewska$^1$ and Lucas Venturato$^{2,3}$\\
  $^1$ Department of Mathematics and Information Sciences, Warsaw University of Technology, Koszykowa 75, 00-662 Warsaw, Poland\\
 $^2$ Departamento de Matem\'atica, FCE, Universidad Austral, Paraguay 1950, S2000FZF Rosario, Argentina \\
 $^3$CONICET, Argentina\\
 (sroscani@austral.edu.ar, K.Ryszewska@mini.pw.edu.pl, lventurato@austral.edu.ar)
                   \vspace{0.2cm}

       \end{center}

\small

\textbf{Abstract:} Taking into account the recent works \cite{RoTaVe:2020}  and  \cite{Rys:2020}, we consider a phase-change problem for a one dimensional material with a non-local flux, expressed in terms of the Caputo derivative, which derives in a space-fractional Stefan problem. We prove existence of a unique solution to a phase-change problem with the fractional Neumann boundary condition at the fixed face $x=0$, where the domain, at the initial time, consists of liquid and solid. Then we use this result to prove the  existence of a limit  solution to an   analogous problem with solid initial domain, when it is not possible to transform the domain into a cylinder.\\
  

\normalsize
\section{Introduction}
This paper is devoted to the analysis of the space-fractional, one-phase, one-dimensional Stefan problem, where in the initial time the domain consists only of a solid phase.
The space-fractional Stefan problems are the free boundary problems (FBP) governed by fractional diffusion equations with non-locality in space. Although in the current literature there is a great variety of fractional models \cite{Hilfer}, according to the different approaches that can be considered in this wide field, in this article we will  focus on the problem proposed by Prof. Voller in  \cite{Vo2011} (see also \cite{GaFaVoPa2020}) where the non-local flux is modeled by the Caputo derivative. 

The idea is the following. Suppose that we are modelling a melting phenomena related to a phase change in an infinite slab, due to heat transfer. We discuss a  one-phase model, so we assume that the solid phase is at a constant melting temperature equal to zero and we are interested only in the liquid phase. We suppose also that all the physical parameters involved are constants  equal to one.  
 
 Let $u=u(x,t)$ and $q=q(x,t)$ be the temperature and the flux at every position $x$ and time $t$. If a relation for these two functions is given by the Fourier's law, which states that the flux is proportional to the gradient of temperature, we have 
\begin{equation}\label{FourierLaw} q_l(x,t)=- u_x(x,t), \end{equation}
where the subscript  $l$ refers to the local character.
However, relation \eqref{FourierLaw} is an empiric relation which  can be replaced, as it was exposed, for example, by Gurtin and Pipkin in 1968 \cite{GuPi:1968} for memory fluxes, or more recently for  non-local fluxes  by Atanackovi\'c et. al. in  \cite{AKO:2012}. We consider here a non-local flux law proposed in \cite{Vo2011}  given by the following relation \begin{equation}\label{q-elegido-0}
 q_{nl}(x,t)=-\frac{1}{\G(1-\al)} \int_0^x \,u_x(p,t)(x-p)^{-\al} \rm{dp}.
\end{equation}
It says that  the flux at every time $t$ and position $x$ is a generalized sum of all the local fluxes at every position between the left extreme of the slab ($x=0$) and the current position, where the ``nearest'' local fluxes  have more relevance than the ``farthest'' ones.

The relation \eqref{q-elegido-0} can be expressed in terms of fractional derivatives (see Definition 2) by 
\begin{equation}\label{q-elegido}
 q_{nl}(x,t)=-  \,D_x^{\al} u(x,t).
\end{equation}

Let us recall the  balance equation which follows from the first principle of thermodynamics
\begin{equation}\label{cont eq}
\frac{\p u}{\p t}(x,t)=-\frac{\p q}{\p x}(x,t).
\end{equation}
Placing the non-local flux \eqref{q-elegido} in \eqref{cont eq} we recover the space fractional diffusion equation 
\begin{equation}\label{gov eq}
 \frac{\p u}{\p t}(x,t)- \frac{\p }{\p x}D_x^{\al}u(x,t)=0
\end{equation}
which is the governing equation in the liquid phase.  
We assume that our melting phenomena is modeled by a sharp interface. Then, the phase change occurs due to a jump of the latent heat between the two phases. More precisely 
\begin{equation}\label{RK-eq}
\lim\limits_{x\nearrow s(t)}q(x,t)=\dot{s}(t)
\end{equation}
and we get the ``fractional Stefan condition''  by placing  \eqref{q-elegido} into \eqref{RK-eq}. Our interest lies in the problem with a  Neumann type boundary condition $(\ref{IFSSP}-ii)$,  given by a prescribed flux at the face $x=0$ in terms of the non-local flux \eqref{q-elegido-0}.
We can present now the main problem that we will study in this paper: Find the pair of functions $u\colon Q_{s,T}\rightarrow \bbR$ and $s\colon [0,T]\rightarrow \bbR$ with sufficient regularity such that
\begin{equation}\label{IFSSP}
\begin{array}{lll}
(i)& u_t(x,t)=\frac{\p}{\p x} \, D_x^{\al} u(x,t) & 0<x<s(t), 0<t<T,\\
(ii) & -D_x^{\al} u(0^+,t)=h(t)\geq 0  &  0<t<T,\\
(iii) & u(s(t),t)=0 & 0<t<T,\\
(iv) & s(0)=0, & \\
(v) & \dot{s}(t)=-\lim\limits_{x\rightarrow s(t)^-} \,D_x^{\al} u(x,t) & 0<t<T,
\end{array}
\end{equation}
where the region  $Q_{{s,T}}$ is defined as\begin{equation}\label{Q_sT}Q_{s.T}:=\{(x,t):0< x< s(t),0< t< T\}.\end{equation}

The literature related to space fractional Stefan  problems is rather scant, here we recall that the self-similar solutions to (\ref{IFSSP}) in terms of special functions have been recently obtained in \cite{RoTaVe:2020}, where in accordance with the work of Voller for the stationary case  \cite{Vo:2014}, the solution verifies that  the advance of the free boundary  is proportional to $t^{\frac{1}{1+\al}}$ instead of being proportional to the square root of $t$, as in classical diffusion. As for the results concerning the space-fractional Stefan problem with the fractional Laplace operator we refer to \cite{delteso} and reference therein. Concerning the linear space-fractional diffusion, modeled by the Caputo derivative, in the cylindrical domain, we recommend the reader the paper \cite{NaRy:2020} where the viscosity solutions are studied and \cite{BaMe:2018}, \cite{RyPhD}, where different types of  boundary conditions are discussed. 

It has to be mentioned that a space-fractional problem where the fractional Neumann condition (\ref{IFSSP} - ii) is replaced by homogeneous condition $u_{x}(0,t) = 0$, and (\ref{IFSSP} - iv) is replaced by $s(0) = b>0$ was recently solved in \cite{Rys:2020}. The approach in such article was based on the transformation of domain into a cylindrical one and the semigroup theory. However, in view of (\ref{IFSSP} - iv) the lines of the  proof in \cite{Rys:2020} cannot be followed. In order to find a solution to  (\ref{IFSSP}) we apply the following strategy. We approximate (\ref{IFSSP}) by the sequence of problems (\ref{FSSP-al-N-2}), where for $b$ we take $1/m$ for $m \in \mathbb{N}$. Here we present the problem~(\ref{FSSP-al-N-2}): Find the pair of functions $u\colon Q_{s,T}\rightarrow \bbR$ and $s\colon [0,T]\rightarrow \bbR$ with sufficient regularity such that 
\begin{equation}\label{FSSP-al-N-2}
\begin{array}{lll}
(i)& u_t(x,t)=\frac{\p}{\p x} \, D_x^{\al} u(x,t) & 0<x<s(t), 0<t<T,\\
(ii) & -D_x^{\al} u(0^+,t)= h(t)\geq 0  &  0<t<T,\\
(iii) & u(s(t),t)=0 & 0<t<T,\\
(iv) & u(x,0)=\uz(x)\geq 0 & 0<x<s(0)=b,\\
(v) & \dot{s}(t)=-\lim\limits_{x\rightarrow s(t)^-} D_x^{\al} u(x,t) & 0<t<T.
\end{array}
\end{equation}
We will refer to problem \eqref{IFSSP} as the \textbf{``$b=0$ case''}, and to problem \eqref{FSSP-al-N-2}  as the \textbf{``$b>0$ case''}.
In order to solve (\ref{FSSP-al-N-2}) we adjust the methods used in \cite{Rys:2020} for the case of non-homogeneus fractional Neumann boundary condition. We will see that this choice of boundary condition has an impact on the regularity of the solution near the left boundary. Hence, the proof has to be significantly modified. Then we will find $b$ - independent energy estimates for a solution to (\ref{FSSP-al-N-2}) and we will obtain the existence of the solution $(u,s)$ to (\ref{IFSSP}) by the appropriate limit passage.

In Section 2 we recall the definitions and properties of fractional integrals and derivatives which will be used in subsequent sections. We also present the properties related to moving boundary problems (MBP) and free boundary problems (FBP). Section 3 is devoted to the proof of the existence of a unique classical solution to (\ref{FSSP-al-N-2}). The proof relies on the semigroup approach and the maximum principles. Firstly, we solve the moving boundary problem (Theorem \ref{existance}), and secondly the fractional Stefan problem  (Theorem \ref{finala}). In Section 4 we establish $b$ - independent estimates for a solutions to (\ref{FSSP-al-N-2}) and we obtain a limit solution $(u,s)$ to \eqref{IFSSP} by choosing a subsequence of solutions $(u^{m},s^{m})$ to problem (\ref{FSSP-al-N-2}) with $b = 1/m$. The optimal regularity of the solutions to \eqref{IFSSP} as well as its uniqueness remains an open problem.

The most remarkable results  obtained in this paper are: $(i)$ An existence and regularity result for a space fractional free boundary problem given  in \textbf{Theorem \ref{finala}}, and $(ii)$ an existence and regularity result for a 
solution to a  space fractional Stefan problem \eqref{IFSSP} given in \textbf{Theorem \ref{bzerofinal}}.

Moreover, the paper provides an expansion of tools from \cite{Rys:2020} for the study of  moving and free boundary problems governed by equation \eqref{gov eq} with prescribed fractional flux at the boundary. More precisely: an extremum principle in Proposition \ref{nonpositivity}, an integral condition equivalent to the fractional Stefan condition in Theorem \ref{condicionintegralNC}, a monotonicity result given in Theorem \ref{desigfronteraN} and energy estimates given in Lemma~\ref{estfirst-Lemma} and Lemma~\ref{szac2-lemma}.

\section{Preliminary results}
\subsection{Basics of fractional calculus}
In this section we will recall the definitions of fractional operators and their most important properties.
Here and henceforth, by $L$ we denote a fixed positive constant.
\begin{defi} Let $L>0$ and $\al \in (0,1)$. 
\begin{enumerate}
	\item For every $f \in L^1(0,L)$ the Riemann-Liouville fractional integral of order $\al$ is defined by  \begin{equation}\label{IntRL} I^{\al}f(x)=\frac{1}{\G(\al)}\int_0^x(x-p)^{\al-1} f(p) \rmd p.\end{equation}
\item For $f$ regular enough, we define the Riemann-Liouville fractional derivative as
 \begin{equation}\label{DRL} \p^{\al}f(x)=\frac{d}{d x}I^{1-\al} f(x)=\frac{1}{\G(1-\al)}\frac{d}{dx}\int_0^x(x-p)^{-\al} f(p) \rmd p\end{equation}

and  the fractional derivative of Caputo as 
 \begin{equation}\label{DC} D^{\al}f(x)=I^{1-\al}f'(x)=\frac{1}{\G(1-\al)}\int_0^x(x-p)^{-\al} f'(p) \rmd p.\end{equation}
 
\end{enumerate}

\end{defi}
\begin{remark}\label{reright}
  We will also make use of 
  the right-side operators for $\al \in (0,1)$ defined as follows
  \[
I^{\al}_{-}f(x) = \frac{1}{\Gamma(\al)}\int_{x}^{L}(p-x)^{\al-1}f(p)dp,
\]
\[
\p^{\al}_{-}f(x) = - \poch (I^{1-\al}_{-}f)(x) \m{ and } D^{\al}_{-}f(x) = - \poch I^{1-\al}_{-}[f(x)-f(L)].
\]
\no Below we will discuss the properties of operators given by (\ref{IntRL}) - (\ref{DC}), but of course analogous results may be obtained  for $I^{\al}_{-}$, $\p^{\al}_{-}$~and~$D^{\al}_{-}$.
\end{remark}

\begin{note}
It is clear that the fractional derivatives are well defined for absolutely continuous functions. In this paper, we will work mainly within the framework of Sobolev spaces. Hence, we will make use of the  characterization of the domain of the Riemann-Liouville fractional derivative in $L^{2}$ - space. 
\end{note}
Before we will recall this result let us introduce the following functional space
\[
{}_{0}H^{1}(0,L):=\{ u \in H^{1}(0,L): u(0) = 0\}.
\]
Then for $\al \in (0,1)$ we denote 
\[
{}_{0}H^{\al}(0,L):=[L^{2}(0,L),{}_{0}H^{1}(0,L)]_{\al}.
\]
This complex interpolation space may be characterized as follows
\[
{}_{0}H^{\al}(0,L)=\left\{
\begin{array}{lll}
H^{\al}(0,L)  & \m{ for } & \al\in \left(0,\frac{1}{2}\right), \\
\left\{ u \in H^{\frac{1}{2}}(0,L): \hd \int_{0}^{L}\frac{|u(x)|^{2}}{x}dx<\infty \right\} & \m{ for } & \al=\frac{1}{2}, \\
\left\{ u \in H^{\al}(0,L): \hd u(0)=0 \right\} & \m{ for } & \al\in \left(\frac{1}{2},1\right).
\end{array}
\right.
\]
Now we are ready to recall the characterization of the domain of the Riemann-Liouville fractional derivative which will be a crucial tool in this paper.
\begin{prop}\label{eq}\cite[Example 2.1]{zachermax}, \cite{Y}
For $\al\in [0,1]$ the operators
$\ia: L^{2}(0,L)\longrightarrow {}_{0}H^{\al}(0,L)$ and $\p^{\al}:{}_{0}H^{\al}(0,L)\longrightarrow L^{2}(0,L)$ are isomorphism and the following
inequalities hold
$$\ca^{-1} \| u \|_{_{0}H^{\al}(0,L)}
\leq \| \p^{\al} u \|_{L^{2}(0,L)}\leq \ca
\| u \|_{_{0}H^{\al}(0,L)} \hd \m{ for } u \in {}_{0}H^{\al}(0,L),
$$

$$
\ca^{-1} \| \ia f  \|_{_{0}H^{\al}(0,L)}
\leq \|  f \|_{L^{2}(0,L)}\leq \ca \| \ia f
\|_{_{0}H^{\al}(0,L)} \hd \m{ for } f  \in L^{2}(0,L).
$$
\end{prop}
From the proposition above one may easily deduce the following corollary.
\begin{coro}\label{eqc}\cite[Corollary 1]{Rys:2020}
For $\al, \beta >0$ there holds $I^{\beta}: {}_{0}H^{\al}(0,L)\rightarrow {}_{0}H^{\al+\beta}(0,L),$
where in the case $\gamma > 1$
\[
_{0}H^{\gamma}(0,L) = \{f \in H^{\gamma}(0,L): f^{(k)}(0) = 0,\hd  k=0,\dots, \lfloor\gamma\rfloor -1,  \hd f^{(\lfloor\gamma\rfloor)} \in {}_{0}H^{\gamma - \lfloor\gamma\rfloor}(0,L) \}.
\]
Furthermore, there exists a positive constant $c$ dependent only on $\al,\beta$ such that for every $f \in {}_{0}H^{\al}(0,L) $
\[
\norm{I^{\beta}f}_{{}_{0}H^{\al+\beta}(0,L)} \leq c \norm{f}_{{}_{0}H^{\al}(0,L)}.
\]
\end{coro}
We will also make use of the following local result.
\begin{lemma}\label{local}\cite[Lemma 4]{Rys:2020}
Let $f\in {}_{0}H^{\al}(0,L)$ for $\al \in (0,1)$ and $\p^{\al}f \in H^{\beta}_{loc}(0,L)$ for $\beta \in (\frac{1}{2},1]$. Then $f \in H^{\beta+\al}_{loc}(0,L)$ and for every $0<\delta <\omega<1$ there exists a positive constant $c=c(\delta,\al,\beta)$ such that
\eqq{
\norm{f}_{H^{\beta+\al}(\delta,\omega)} \leq c(\norm{f}_{{}_{0}H^{\al}(0,\omega)} + \norm{\p^{\al}f}_{H^{\beta}(\frac{\delta}{2},\omega)}).
}{nowelc}
\end{lemma}
In the following proposition we collect the basic properties of fractional operators.
\begin{prop}\label{propo frac} The following properties involving the fractional integrals and derivatives of order $\al \in (0,1)$ hold:
\begin{enumerate}
\item \label{RL inv a izq de I} For every $f \in L^1(0,L)$,  
$$\p^{\al} I^{\al}f=f  \quad a.e. \text{ in } \, (0,L).$$

\item  For every $f \in AC[0,L]$,    
$$ I^{\al}(\p^{\al}f)(x)=f(x) - \dfrac{I^{1-\al}f(0^+)}{\G(\al)x^{1-\al}} \quad \quad \forall \, x \in [0,L]. $$

\item\label{propClave} For every  $f \in AC[0,L]$ such that $I^{1-\al}f' \in AC[0,L]$ it holds that 
$$\frac{d}{d x}D^\alpha f=\p^\alpha f', \qquad a.e. \text{ in } (0,L).$$
\item\label{propClave2} For every  $f \in AC[0,L]$ 
$$\p^{1-\alpha}D^{\al} f=f', \qquad a.e. \text{ in } (0,L).$$

\item\label{Proposition 3.1} If $f\in H^{1+\al}(0,L)$and $\al \in (0,1)$, then $D^\al f(0) = 0$. 
\end{enumerate}
\end{prop}
\begin{proof} See Chapter 2 in \cite{Samko} for $1-4.$ and \cite{NaRyVo:2021} for $\ref{Proposition 3.1}$.
\end{proof}
The next proposition establishes the relation between the regularity of a function and the value of its Caputo derivative at the origin.
\begin{prop}\cite[Lemma 3.11]{RyPhD}\label{Lemma3.11-Phd-K}. Let $f \in  AC[0,L]$ and let us denote $D^\al f(0):=\lim\limits_{x\rightarrow 0}D^\al f(x)$. Then:
\begin{enumerate}
	\item If $D^\al f(0)$ exists and  $D^\al f(0)=c$, then $\lim\limits_{x\rightarrow 0} \frac{f(x)-f(0)}{x^\al} = \frac{c}{\G(1+\al)}.$
\item If the limit $\lim\limits_{x\rightarrow 0} \frac{f'(x)}{x^{\al-1}}$  exists and it is equal to $\frac{c}{\G(\al)}$, then $D^\al f(0)=c$.
\end{enumerate}
\end{prop}
Here we present a modified version of \cite[Lemma 6 and Lemma 7]{Rys:2020}. Now we write the regularity assumptions in terms of fractional Sobolev spaces which seems to be more natural. In particular for small $\al$ we do not demand the existence of the first derivative and of the second derivative, in Proposition \ref{extremumPple} and Proposition \ref{nonpositivity}, respectively. We present the proof only of Proposition \ref{nonpositivity} because it is more complex. Proposition \ref{extremumPple} may be proven analogously. 

\begin{prop}\label{extremumPple} Let us assume that $f:[0, L]\rightarrow \mathbb{R}$ is absolutely continuous on $[0,L]$ and there exists $\beta > \frac{1}{2}$ such that $f \in H^{\al+\beta}(\ve,L)$ for every $\ve > 0$.
Then $D^{\al}f$ is continuous on $(0,L]$ and

\begin{enumerate}
	\item If f attains its maximum over the interval $[0,L]$ at the point $x_0 \in (0,L]$, then $D^\al f(x_0)\geq 0$. Furthermore, if $f$ is not constant on $[0, x_0]$, then $D^\al f(x_0) > 0$.
\item If f attains its minimum over the interval $[0,L]$ at the point $x_0 \in (0,L]$, then $D^\al f(x_0)\leq 0$. Furthermore, if f is not constant on $[0, x_0]$, then $D^\al f(x_0) <0$.	
\end{enumerate}
\end{prop}

\begin{prop}\label{nonpositivity}
Let $\al \in (0,1)$, $f:[0,L]\rightarrow \mathbb{R}$ be an absolutely continuous function such that there exists $\beta > \frac{1}{2}$ such that $f' \in H^{\al+\beta}(\ve,L)$ for every $\ve > 0$. Then $\frac{d}{dx} D^\al f$  is continuous on $(0,L]$ and
\begin{enumerate}
  \item if $f$ attains its local maximum at $x_{0}\in (0,L)$ which is a global maximum on $[0,x_{0}]$, then   $(\frac{d}{dx} D^{\al} f)(x_{0})~\leq~0$. Furthermore, if $f$ is not constant on $[0,x_{0}]$, then $(\frac{d}{dx} D^{\al} f)(x_{0})~<~0$.
  \item If $f$ attains its local minimum at $x_{0}\in (0,L)$ which is a global minimum on $[0,x_{0}]$, then \linebreak  $\left(\frac{d}{dx} \da f\right)(x_{0})~\geq~0$. Furthermore, if $f$ is not constant on $[0,x_{0}]$, then $\left(\frac{d}{dx} D^{\al} f\right)(x_{0}) > 0$.
\end{enumerate}
\end{prop}
\begin{proof}
Let us begin with the proof of continuity of $\frac{d}{dx} D^{\al} f$. To this end, we take $x_{1}, x \in (0,L)$. Let us assume that $x_{1} < x$. The case $x < x_{1}$ may be shown analogously. We note that for every $0<\ve<y<L$ there holds
\[
\Gamma(1-\al)\left(\frac{d}{dx} D^{\al} f\right)(y)= \frac{d}{d y} \int_{0}^{\ve} (y-p)^{-\al}f'(p)\rmd p + \frac{d}{d y} \int_{\ve}^{y} (y-p)^{-\al}f'(p)\rmd p
\]
\[
=-\al \int_{0}^{\ve} (y-p)^{-\al-1}f'(p)\rmd p +  \frac{d}{d y} \int_{\ve}^{y} (y-p)^{-\al}[f'(p)-f'(\ve)]\rmd p + f'(\ve)(y-\ve)^{-\al}.
\]
Let us denote $\p^{\al}_{\ve}f(x) := \frac{1}{\Gamma(1-\al)} \frac{d}{dx} \int_{\ve}^{x}(x-p)^{-\al}f(p)\rmd p$. Then, taking arbitrary $\ve \in (0,x_{1})$ we obtain that
\[
\Gamma(1-\al)\abs{\frac{d}{dx} D^{\al} f(x) - \frac{d}{dx} D^{\al} f(x_{1})}
\leq  \al \int_{0}^{\ve} [(x_{1}-p)^{-\al-1} - (x-p)^{-\al-1}]\abs{f'(p)}\rmd p 
\]
\[
  +\Gamma(1-\al) \abs{\p^{\al}_{\ve}[f'-f'(\ve)](x) -\p^{\al}_{\ve}[f'-f'(\ve)](x_{1}) } + \abs{f'(\ve)}[(x_{1}-\ve)^{-\al}-(x-\ve)^{-\al}].
\]
The first term tends to zero as $x\rightarrow x_{1}$ because the convergence under the integral is uniform. By Corollary \ref{eqc} we have
\[
I^{1-\al}_{\ve}[f'-f'(\ve)] = \frac{1}{\Gamma(1-\al)} \int_{\ve}^{x}(x-p)^{-\al}[f'-f'(\ve)]\rmd p \in {}_{0}H^{1+\beta}(\ve,L),
\]
Hence, we obtain that $\p^{\al}_{\ve}[f'-f'(\ve)] \in H^{\beta}(\ve,L)\hookrightarrow C[\ve,L]$.
Thus, the continuity of $\frac{d}{dx} D^{\al} f$ on $(0,L)$ is proven.
We will prove only the part of the claim concerning maximum, because the proof of the second part of the claim is analogous. We define $g(x) := f(x_{0})-f(x)$. Then $g$ is non negative on $[0,x_{0}]$, $g'(x_{0})=0$ and $\frac{d}{dx} D^{\al} g = - \frac{d}{dx} D^{\al} f$. We note that by the Sobolev embedding $g' \in C^{0,\gamma}[\ve,L]$ for $\gamma = \al+\beta-\frac{1}{2} > \al$ for every $\ve>0$. Hence,
for every $0<\ve <x \leq x_{0}$ we may estimate as follows
\eqq{
\abs{g'(x)} = \abs{g'(x)- g'(x_{0})} \leq c \abs{x-x_{0}}^{\gamma}
}{gp}
and
\eqq{
g(x) \leq \int_{x}^{x_{0}}\abs{ g'(p)}dp \leq \frac{c}{\gamma+1}\abs{x-x_{0}}^{\gamma+1}.
}{gbezp}
Making use of these estimates we may differentiate under the integral sign as follows
\[
\left(\frac{d}{dx} D^{\al} g\right)(x_{0}) = \frac{1}{\Gamma(1-\al)}\left( \frac{d}{dx} \int_{0}^{x}(x-p)^{-\al}g'(p)\rmd p\right)(x_{0})
\]
\[
=\frac{1}{\Gamma(1-\al)}\lim_{p\rightarrow x_{0}^{-}}(x_{0}-p)^{-\al}g'(p)
-\frac{\al}{\Gamma(1-\al)}\int_{0}^{x_{0}}(x_{0}-p)^{-\al-1}g'(p)\rmd p.
\]
and the  limit is equal to zero by the estimate (\ref{gp}).
Applying integration by parts we obtain further
\[
\left(\frac{d}{dx} D^{\al} g\right)(x_{0}) = -\frac{\al}{\Gamma(1-\al)}\int_{0}^{x_{0}}(x_{0}-p)^{-\al-1}g'(p)\rmd p =
 -\frac{\al}{\Gamma(1-\al)}\lim_{p\rightarrow x_{0}^{-}}(x_{0}-p)^{-\al-1}g(p)
 \]
 \[
 + \frac{\al}{\Gamma(1-\al)}x_{0}^{-\al-1}g(0) + \frac{\al(\al+1)}{\Gamma(1-\al)}\int_{0}^{x_{0}}(x_{0}-p)^{-\al-2}g(p)\rmd p.
\]
By (\ref{gbezp}) the limit equals zero, hence we arrive at
\[
\left( \frac{d}{dx} D^{\al} g\right)(x_{0}) = \frac{\al}{\Gamma(1-\al)}x_{0}^{-\al-1}g(0) + \frac{\al(\al+1)}{\Gamma(1-\al)}\int_{0}^{x_{0}}(x_{0}-p)^{-\al-2}g(p)\rmd p
\]
and
\[
\left(\frac{d}{dx} D^{\al} g\right)(x_{0}) \geq 0, \m{ which implies } \left(\frac{d}{dx} D^{\al} f\right)(x_{0}) \leq 0.
\]
Moreover, from the formula above, we obtain that if $f$ is not a constant function on $[0,x_{0}]$ then $\left(\frac{d}{dx} D^{\al} f\right)(x_{0})~<~0$.
\end{proof}

\subsection{The analysis of Space Fractional MBP and FBP}
We now return to problems \eqref{IFSSP} and \eqref{FSSP-al-N-2}.
In this subsection we present some results related to problems \eqref{IFSSP} and \eqref{FSSP-al-N-2}  that will be used later. Most of the results presented below have their analogous version for the classical case ($\al = 1$) that can be found in classical literature related to Stefan problems such as \cite{Cannon} or \cite{Fridman:1958}.

From Proposition \ref{propo frac} item \ref{Proposition 3.1} we know that the fractional flux at the left endpoint of the interval vanishes for functions smooth enough. Thus, if we impose a boundary condition characterized by a non-zero function $h$, we expect to obtain  solutions with less regularity. Furthermore, from Proposition \ref{Lemma3.11-Phd-K} we may expect that the solution behaves as $h(t)\frac{x^{\al}}{\Gamma(1+\al)}$ as $x$ approaches zero.

Particularly, to be consistent throughout this paper, we assume that
\begin{equation}\label{zalh}
h \, \text{is a non-negative function in }\,  W^{1,\infty}(0,T). \end{equation}

Let us give the definitions of the solution to (\ref{FSSP-al-N-2}) and \eqref{IFSSP}.
\begin{defi}\label{sol-b>0}
A pair of functions $(u,s)$ is a solution to \eqref{FSSP-al-N-2}, if:
\begin{enumerate}
	\item $u\in C(\overline{Q_{s,T}})$, $ u_t$ and $ \frac{\partial}{\partial x} \, ^C_0D_x^{\al}u\in C(Q_{s,T})$.
	\item  $s$ satisfies that
\begin{equation}\label{zals}
s\in C^{0,1}([0,T]), \exists M>0 \text{ such that } 0\leq \dot{s}(t)\leq M \text{ a.e. } t\in [0,T].\end{equation}
\item $(u,s)$ verifies (\ref{FSSP-al-N-2} i-iv) in the classical sense and (\ref{FSSP-al-N-2}-v) almost everywhere.
\end{enumerate}
 \end{defi}
\begin{defi}\label{regularidad-solucion-kasia-2}
A pair of functions $(u,s)$ is a limit solution to  \eqref{IFSSP}, if there exists a sequence $(u_b,s_b)_{b>0}$ of classical solutions to problem \eqref{FSSP-al-N-2} converging to $(u,s)$ when $b$ tends to $0$ such that:
\begin{enumerate}
	\item $u \in  C(Q_{s,T})$ is such that (\ref{IFSSP}-i) is satisfied in $L^{2}(0,s(t))$ for almost every $t \in (0,T)$.
	\item  $s$ satisfies (\ref{zals})
\item $(u,s)$ fulfills (\ref{IFSSP}-ii) and (\ref{IFSSP}-v) almost everywhere and (\ref{IFSSP}-iii) everywhere on (0,T].
\end{enumerate}
 \end{defi}


\begin{theo}\label{condicionintegralNC}
Let  the pair $(u,s)$ be a solution to \eqref{FSSP-al-N-2} such that $u_t$ is integrable.  Then, the fractional Stefan condition
\begin{equation}\label{derivs}
\dot{s}(t)=- \,D_x^{\al} u(s(t),t), \quad\quad 0<t<T,
\end{equation}
in \eqref{FSSP-al-N-2}, can be replaced by an equivalent integral condition given by
\begin{equation}\label{CondInt1N}
s(t)=b+\int_0^t h(\tau) \rmd \tau+\int_0^b \uz(x)\rmd x-\int_0^{s(t)}u(x,t)\rmd x, \quad\quad 0<t<T. 
\end{equation}

\end{theo}

\begin{proof} 

Integrating  equation $(\ref{FSSP-al-N-2}-i)$ from 0 to $s(t)$ we obtain
\begin{equation}
\, D_x^{\al} u(s(t),t)+h(t)=\int_0^{s(t)} u_t(x,t)\rmd x.
\end{equation}
Now by integrating first respect of $t$, taking into account \eqref{zalh} and applying Fubini's theorem  we get 
\begin{equation}\label{Equiv-1}
\begin{split}
\int_0^t D_x^{\al} u(s(\tau),\tau)\rmd\tau &=-\int_0^t h(\tau)\rmd \tau +\int_0^t\int_0^{s(\tau)} u_t(x,\tau)\rmd x\rmd\tau\\
&=-\int_0^t h(\tau)\rmd \tau +\int_0^b\int_0^t u_t(x,\tau)\rmd\tau\rmd x+\int_b^{s(t)}\int_{s^{-1}(x)}^t u_t(x,\tau)\rmd\tau\rmd x\\
&=-\int_0^t h(\tau)\rmd \tau +\int_0^b [u(x,t)-u(x,0)]\rmd x+\int_b^{s(t)} [u(x,t)-0]\rmd x\\
&=-\int_0^t h(\tau)\rmd \tau +\int_0^{s(t)} u(x,t)\rmd x-\int_0^b\uz(x)\rmd x.
\end{split}
\end{equation}

Finally, we get \eqref{CondInt1N} by  replacing the fractional Stefan condition \eqref{derivs} at the left side of \eqref{Equiv-1}. 

Let us now suppose that \eqref{CondInt1N} holds and differenciate both sides. Then we get
\begin{equation}\label{Equiv-2}
\begin{split}
\dot{s}(t) &= h(t) -\int_0^{s(t)} u_t(x,t)\rmd x - \lim\limits_{x\rightarrow s(t)}u(x,t) \dot{s}(t)
\end{split}
\end{equation}
and using equation $(\ref{FSSP-al-N-2}-i)$ again we recover the  fractional Stefan condition $(\ref{FSSP-al-N-2}-v)$ . 
%
%

\end{proof}
\begin{defi} Let $s:[0,T]\rightarrow \bbR$ be a function verifying \eqref{zals} and let $Q_{s,T}$ be the region defined in \eqref{Q_sT}. The parabolic boundary of $Q_{s,T}$ is defined by 
\begin{equation}\label{ParBoun}\p \G_{s,T}:= \left\{(0,t): \,0\leq t < T  \right\}\cup \left\{(x,0): \, 0\leq x\leq b\right\}\cup \left\{(t,s(t)) : 0\leq t< T\right\}. \end{equation}
\end{defi}

The next results holds for moving-boundary problems (MBP), that is, when the right boundary is a known function an we look for a solution $u$ such that 

\begin{equation}\label{niech}
  \begin{array}{lll}
(i) & u_{t} - \poch \da u = 0 & \textrm{ in }  Q_{s,T}, \\
(ii) & -\da u(0,t) = h(t)\geq 0, & 0<t<T \\
(iii) & u(s(t),t) = 0 & t \in (0,T), \\
(iv) & u(x,0) = \uz(x)\geq 0 & \textrm{ for } 0<x<b 
\end{array} \end{equation}
with  given functions $s$, $h$ and $\uz$, where $s$ and $h$ verifies \eqref{zals} and \eqref{zalh} respectively.

The next Theorem is a fundamental tool in this article. Notice that it has been presented before in \cite[Lemma 8]{Rys:2020} by assuming different regularity of a solution. However, the same proof can be rewritten by considering the novel regularity assumption presented in this paper in Proposition \ref{nonpositivity}.


\begin{theo}\label{weakPrinc}{(Weak Extremum Principle).} Let $u$ be a solution  to the  equation
$$ {u}_t-\frac{\p}{\p x} D_x^{\al} u=f\quad \text{in }\, Q_{s,T} $$
having the regularity given in Definition \ref{sol-b>0},  such that  $u(\cdot,t)\in AC[0,s(t)]$ $\forall t\in (0,t)$. Also suppose that $u$ verifies the local regularity condition 
\eqq{
\exists \, \beta > \frac{1}{2} \m{ such that for every } t\in (0,T) \m{ and every } 0<\ve<\omega < s(t), \hd  u_{x}(\cdot,t) \in H^{\al+\beta}(\ve,\omega).
}{regumax}
Then $u$ attains its maximum on $\p \G_{s,T}$ if $f\leq 0$, and   $u$ attains its minimum on $\p \G_{s,T}$ if $f\geq 0$.\end{theo}

Now we can state the comparison principle. 
\begin{prop}\label{comparacion}
Let us assume that $h$ and $s$ satisfy (\ref{zalh}) and (\ref{zals}), respectively. Further, let $v$ satisfy the regularity assumptions stated in Theorem \ref{weakPrinc} and

\begin{equation}\label{WMP}
\begin{array}{cll}
(i) & {v}_t(x,t)-\frac{\p}{\p x} D_x^{\al} v(x,t)\geq0 & 0<x<s(t), 0<t<T,\\
(ii) & -\, D_x^{\al} v(0,t)=h(t)\geq 0  &  0<t<T,\\
(iii) & v(s(t),t)\geq 0 & 0<t<T,\\
(iv) & v(x,0)\geq 0 & 0<x<s(0)=b,\\

\end{array}
\end{equation}

where the condition $(iv)$ vanishes if $b=0$. Then, $v(x,t)\geq 0 \text{ on } Q_{s.T}.$
\end{prop}

\begin{proof}
For each $\ve > 0$ define $w^\varepsilon (x,t)=v(x,t)+\varepsilon\left(l - \frac{x^\al}{\Gamma(1+\alpha)}\right)$, where $\displaystyle l>\sup_{0\leq t \leq T}\left\{\frac{s^\al(t)}{\Gamma(1+\al)}\right\}$.\\
 Then,
$$w^\varepsilon_t(x,t) - \frac{\p}{\p x} D_x^{\al} w^\varepsilon (x,t)= {v}_t(x,t)-\frac{\p}{\p x} D_x^{\al} v(x,t) \geq 0$$
and by Theorem \ref{weakPrinc}, $w^\varepsilon$ attains its minimum on $\p \G_{s,T}$.

On the other hand,
\begin{equation}\label{nonnegat1}
-\,D_x^{\al} w^\varepsilon(0,t)=h(t)-\varepsilon \cdot(-1)\geq \varepsilon> 0,
\end{equation}
\begin{equation}\label{nonnegat2}
w^\varepsilon(x,0)=v(x,0)+\varepsilon\left(l - \frac{x^\al}{\Gamma(1+\alpha)}\right)> 0,
\end{equation}
\begin{equation}\label{nonnegat3}
w^\varepsilon(s(t),t)=v(s(t),t)+\varepsilon\left(l - \frac{s^\al(t)}{\Gamma(1+\alpha)}\right)> 0.
\end{equation}
Furthermore, for every fixed $t$ we can apply Proposition \ref{Lemma3.11-Phd-K} to $w^\varepsilon(\cdot,t)$  to obtain

\begin{equation}
\lim\limits_{x\rightarrow 0} \frac{w^\varepsilon (x,t)-w^\varepsilon(0,t)}{x^\al}=-\frac{h(t)+\varepsilon}{\Gamma(1+\al)}< 0.
\end{equation}

Then, $w^\varepsilon$ does not attain its minimum at the left boundary and we can deduce that $w^\varepsilon> 0$.

Letting $\varepsilon$ tend to zero, we see that $v(x,t)\geq 0$ in $Q_{s.t}$.

\end{proof}
We recall a useful lemma \cite[Lemma 4.15]{RyPhD} from which the strict increase of the free boundary will be deduced.
\begin{lemma}\label{Lemma 4.15}\cite[Lemma 4.15]{RyPhD} Let $u$ be a solution to \eqref{niech} verifying the assumptions of Theorem \ref{weakPrinc} and let $t_0 \in (0,T]$ be fixed such that  $u(s(t_0),t_0)=0$.  Then either $D_x^{\al}u(s(t_0,t_0)<0$ or $u \equiv 0$ on $Q_{s,t_0}$.
\end{lemma}
By  Lemma \ref{Lemma 4.15} and Proposition \ref{comparacion} it is not difficult to obtain the following result.
\begin{prop}\label{nonposda}  Let $u$ be a solution to \eqref{niech} such that $u$ verifies the assumptions of Theorem~\ref{weakPrinc}. 
Then $(\da u)(s(t),t)~\leq~0$. Furthermore, if $\uz \not\equiv 0$ or $h \not\equiv 0$, then for every $t \in (0,T]$ we have $(\da u)(s(t),t)~<~0$.
\end{prop}
\begin{proof}
Applying Proposition \ref{comparacion} to $u$ we deduce that $u \geq 0$ on $\overline{Q_{s,T}}$. Since $u(s(t),t)=0$ we  deduce that $u$ attains its minimum at every point $(s(t),t)$ for every $t\in (0,T)$. Then, from the extremum principle given in  Proposition \ref{extremumPple} it holds that $\da u(s(t),t) \leq 0$. Furthermore, if $\uz \not \equiv 0$, from  Lemma \ref{Lemma 4.15} we obtain that $\da u(s(t),t) < 0$.
\end{proof}

\begin{prop}\label{daboundlem}
Let $u$ be a solution to \eqref{niech} which verifies the assumptions of Theorem~\ref{weakPrinc}. Let us assume that $h$ and $s$ satisfy repsectively \eqref{zalh} and  \eqref{zals}, and $\uz$ is such that $0\leq \uz(x)\leq \frac{M}{\G(1+\al)}(b^\al-x^\al)$, where $M$ is the constant in \eqref{zals} fixed as $M=||h||_{\infty}$. Then the following bounds hold
\eqq{(\da u)(s(t),t) \geq -M \m{ for every } t \in (0,T)}{dabound}
and
\eqq{0 \leq u(x,t)\leq  \frac{M}{\Gamma(1+\al)} (s(t)^{\al}-x^\al) \m{ for } (x,t) \in Q_{s,T}.}{estu}
\end{prop}
\begin{proof} Let $u$ be a solution to problem \eqref{niech}, $v(x,t)=\frac{M}{\G(1+\al)}(s(t)^\al-x^\al)$ and define $w(x,t):=v(x,t)-u(x,t)$. We will prove that $w$ satisfies \eqref{WMP} so we can apply Proposition \ref{comparacion}.

Note that function $v$ verifies that ${v}_t(x,t)-\frac{\p}{\p x} D_x^{\al} v(x,t)=\frac{\al M}{\Gamma(1+\al)}s^{\al-1}(t)\dot{s}(t) \geq 0$  due to \eqref{zals}. Also $v(s(t),t)=0$, $D_x^\al v \equiv-M \leq h(t)$ because $M=||h||_\infty$, and from assumption $v(x,0)=\frac{M}{\Gamma(1+\al)} (b^{\al}-x^\al)\geq \uz(x)$. We deduce then that $(v-u)(x,t)\geq 0$ and \eqref{estu} holds. Finally, applying Proposition \ref{nonposda} we get \eqref{dabound}.
\end{proof}
Finally, we present the results regarding the monotone dependence upon data. In the proof we follow the lines of the proof of \cite[Theorem 5]{Rys:2020}. 
\begin{theo}\label{desigfronteraN}
Let $(s_1,u_1)$ and $(s_2,u_2)$ be the solutions to \eqref{FSSP-al-N-2} in the sense of Definition \ref{sol-b>0}, satisfying the regularity assumptions from Theorem \ref{weakPrinc}, corresponding, respectively, to the data $(b_1,{u_{0}}_1,h_1)$ and $(b_2,{u_0}_2,h_2)$ such that ${u_0}_2(b_2)=0$. Then, if $b_1\leq b_2$, ${u_0}_1\leq {u_0}_2$, $h_1\leq h_2$, it holds that $s_1\leq s_2$.

\end{theo}

\begin{proof}

\textsl{Case 1:} $b_1<b_2$ and ${u_0}_1\not\equiv {u_0}_2$ on $[0,b_1]$. We will proceed by contradiction.
Suppose that $s_1\not\leq s_2$. Then, we can define $t_0=\inf\{t\in[0,T]:s_1(t)=s_2(t)\}$. 
In fact, $s_1$ and $s_2$ are increasing functions such that $s_1(0)=b_1<b_2=s_2(0)$ and $s_1(\overline{t})>s_2(\overline{t})$ for some $\overline{t}\in (0,T)$,
hence $\{t\in[0,T]:s_1(t)=s_2(t)\}$ is a nonempty and bounded set and $t_0$ is well defined. Furthermore, since $s_1(0)<s_2(0)$, we have $t_0 \neq 0$.


Defining $v:=u_2-u_1$, we have
$$v(s_1(t_0),t_0)=u_2(s_1(t_0),t_0)-u_1(s_1(t_0),t_0)=u_2(s_2(t_0),t_0)-u_1(s_1(t_0),t_0)=0,$$
and by Proposition~\ref{comparacion}, $v\geq 0$ on $Q_{s_1,t_0}$. Now, taking into account that ${u_0}_1\not\equiv {u_0}_2$ we apply   Proposition~\ref{nonposda} to get that $D_x^{\al}v(s_1(t_0),t_0)<0$. Hence,  $ D_x^{\al}u_2(s_1(t_0),t_0)-D_x^{\al}u_1(s_1(t_0),t_0)=-\dot{s}_2(t_0)+\dot{s}_1(t_0)<0.$ 


The last assertion is  equivalent to
\begin{equation}\label{mono-1}\dot{s}_1(t_0)<\dot{s}_2(t_0).\end{equation}
 On the other hand, by definition of $t_0$, $s_1<s_2$ on $[0,t_0)$ and $s_1(t_0)=s_2(t_0)$. Then,
 $$\dot{s}_1(t_0)=\lim\limits_{t\rightarrow t_0^-}\frac{s_1(t)-s_1(t_0)}{t-t_0}\geq \lim\limits_{t\rightarrow t_0^-}\frac{s_2(t)-s_2(t_0)}{t-t_0}=\dot{s}_2(t_0),$$
 which contradicts \eqref{mono-1}, and we conclude that $s_1\leq s_2$.\\



\textit{Case 2:} $b_1\leq b_2$.\\
 Let us fix $\delta>0$ and consider $b_\delta = b_2+\delta$ and $h_\delta\equiv h_2$. Also, let ${u_0}_\delta$ be a smooth function defined on $[0,b_2+\delta]$ such that
$${u_0}_\delta \equiv 0 \text{ on } [b_\delta-\delta/2, b_\delta], \quad \quad  {u_0}_\delta \geq {u_0}_2 \text{ on } [0, b_2], $$
$$\max_{x\in [0,b_2]}\{{u_0}_\delta (x)-{u_0}_2(x) \}=\delta,\quad\quad \max_{x\in [b_2,b_\delta-\delta/2]}\{{u_0}_\delta (x)\}\leq \delta.$$

 Let $(u_\delta,s_\delta)$ be a solution to \eqref{FSSP-al-N-2} corresponding to the data $(b_\delta,{\uz}_\delta,h_\delta)$. We can apply \textsl{Case 1} to $(b_1,{u_0}_1,h_1)$ and $(b_\delta,{u_0}_\delta,h_\delta)$, and also to $(b_2,{u_0}_2,h_2)$ and $(b_\delta,{u_0}_\delta,h_\delta)$, obtaining respectively that  $s_1\leq s_\delta$ and  $s_2\leq s_\delta$. Now we apply Theorem \ref{condicionintegralNC} to get
\begin{equation}
s_\delta(t)=b_\delta+\int_0^t h_\delta(\tau) \rmd \tau+\int_0^{b_\delta} {u_0}_\delta(x)\rmd x-\int_0^{s_\delta(t)}u_\delta(x,t)\rmd x,
\end{equation}
 
\begin{equation}
s_2(t)=b_2+\int_0^t h_2(\tau) \rmd \tau+\int_0^{b_2} {u_0}_2(x)\rmd x-\int_0^{s_2(t)}u_2(x,t)\rmd x.
\end{equation}
 
 Hence
 \begin{equation}\label{eq000}
\begin{split}
s_\delta(t)-s_2(t)&=\delta+\int_0^{b_2} \hspace{-0.3cm}\left({u_0}_\delta(x)-{\uz}_2(x)\right)\rmd x+\int_{b_2}^{b_\delta}\hspace{-0.3cm} {\uz}_\delta(x)\rmd x
-\int_0^{s_2(t)}\hspace{-0.5cm}\left( u_\delta(x,t)-u_2(x,t)\right)\rmd x-\int_{s_2(t)}^{s_\delta(t)}\hspace{-0.6cm}u_\delta(x,t)\rmd x\\
&\leq \delta+b_2\delta+\frac{\delta^2}{2}-\int_0^{s_2(t)}\left( u_\delta(x,t)-u_2(x,t)\right)\rmd x-\int_{s_2(t)}^{s_\delta(t)}u_\delta(x,t)\rmd x.
\end{split}
\end{equation}
 By Proposition \ref{comparacion}, $u_2\leq u_\delta$ on $Q_{s_2,t}$ and $u_\delta\geq 0$ on $Q_{s_\delta,t}$. Thus, from \eqref{eq000} we obtain that
 \begin{equation}
s_\delta(t)-s_2(t)\leq \delta+b_2\delta+\frac{\delta^2}{2} \quad \text{for every } t\in (0,T).
\end{equation}
Thus, we deduce that, for arbitrary $\delta >0$ it holds that 
$$s_1 (t)\leq s_\delta (t)\leq s_2(t) +\delta+ b_2\delta+\frac{\delta^2}{2}.$$
Passing to the limit with $\delta$ we obtain the claim. 
\end{proof}
We finish this section with the uniqueness result.

\begin{theo}\label{uniN}
If problem \eqref{FSSP-al-N-2} admits a solution in the sense of Definition \ref{sol-b>0}, satisfying the regularity assumptions from Theorem \ref{weakPrinc}, then it is unique.
\end{theo}

\begin{proof}

Let $(s_1,u_1)$ and $(s_2,u_2)$ be solutions to \eqref{FSSP-al-N-2} corresponding to the same data $(b,\uz,h)$. Then, by Theorem \ref{desigfronteraN}, $s_1\leq s_2$ and $s_2\leq s_1$ imply that $s_1\equiv s_2$. Applying Proposition \ref{comparacion} to $u:=u_2-u_1$ and $v=u_1-u_2$, we obtain respectively $u_1\leq u_2$ and $u_2\leq u_1$. Hence, $u_1 \equiv u_2$.

\end{proof}

\section{A solution for the b>0 case.}

In this section we will work with the case in which, at the initial time, the material presents the liquid and solid state delimited by the positive position $s(0)=b$. The positiveness of $b$ is the key that will allows us to transform the region $Q_{s,T}$  into a cylindrical domain.
The strategy of the proof is the same as in \cite{Rys:2020}. Firstly, we assume that function $s$ is given and we find a solution $u$ to the MBP (\ref{niech}). This is done by a transformation of the domain and the application of the evolution operator theory. Then we apply the Schauder fixed point theorem to obtain the solution to the FBP (\ref{IFSSP}). 
\subsection{The transformed problem and the characterization of the solution.}

Firstly, we will look for a solution $u$ to \eqref{niech}, assuming that the known boundary $s$ verifies \eqref{zals}. In order to pass to the cylindrical domain, we apply the standard substitution 
\begin{equation}\label{sust}
p = \frac{x}{s(t)}\end{equation}
 and, by defining
\eqq{w(p,t) := u(s(t)p,t) = u(x,t),}{defv}
we transform \eqref{niech} into the following problem in the cylindrical domain $(0,1)\times (0,T)$
\begin{equation}\label{niecw}
 \begin{array}{lll}
(i)& w_{t} - x\frac{\dot{s}(t)}{s(t)}w_{x} - \frac{1}{s^{1+\al}(t)} \poch \da w = 0 &  0<x<1, \hd 0<t<T,\\
(ii) & -\da w(0,t) = s^{\al}(t)h(t),& 0<t<T, \\ 
(iii) & w(1,t) = 0 &  0<t<T, \\
(iv) & w(x,0) = w_{0}(x) &  0<x<1, \\
\end{array}  \end{equation}
where $w_{0}(x) = u_{0}(xb)$ and we have renamed $p$ by $x$.

 
  We will find a solution to problem \eqref{niecw} applying the theory of evolution operators, following the lines proposed in \cite{RyPhD}. For that purpose, we recall the result given in \cite{ja}.
\begin{theo}\label{analitycznosc} \cite[Theorem 2]{ja}
The operator $\poch \da: \dd \subseteq \ld \rightarrow \ld$ is a generator of analytic semigroup where 
\eqq{D\left(\poch \da\right) \equiv \dd := \{u \in H^{1+\al}(0,1) :u_{x} \in {}_{0}H^{\al}(0,1),\hd  u(1)=0\},}{dziedzina}
\[\norm{u}_{\dd} = \norm{u}_{ H^{1+\al}(0,1)} \m{ for } \al \in (0,1)\setminus \{\frac{1}{2}\},
\]
\[
\norm{u}_{\dd} = \left(\norm{u}^{2}_{ H^{\frac{3}{2}}(0,1)} + \izj \frac{\abs{u_{x}(x)}^{2}}{x}dx \right)^{\frac{1}{2}}  \m{ for } \al = \frac{1}{2}.
\]
\end{theo}

We note that, in \cite{Rys:2020}, by using Theorem \ref{analitycznosc}, a similar problem to \eqref{niecw}, but adressed with a classical Neumann homogeneus boundary condition, was  analyzed. 
Hence, in order to use the result in Theorem~\ref{analitycznosc}, but taking into account our different prescribed boundary condition and Proposition \ref{Lemma3.11-Phd-K}, we will search for the solution to \eqref{niecw} represented as a sum of two functions, one of them destined to belong to $\dd$ with respect to the space for every $t>0$.   Let $\varphi$ be a smooth function such that \eqq{\varphi(x) = \frac{x^{\al}}{\Gamma(\al+1)} \quad \text{on}\quad  [0,\frac{1}{2}],\quad \text{and}\quad  \varphi \equiv  0\quad   \text{on } \, [\frac{3}{4},1]}{auxvar} and define the function 
\begin{equation}\label{descomp-w}
w(x,t) := v(x,t) + h(t)s^{\al}(t)\varphi(x).
\end{equation}
We note that $w$ is a solution to  (\ref{niecw}) if and only if  $v$ is a solution to the \textsl{$\varphi$-auxiliary problem} 
\begin{equation}\label{niecv}
 \begin{array}{lll}
(i) & v_{t} - x\frac{\dot{s}(t)}{s(t)}v_{x} - \frac{1}{s^{1+\al}(t)} \poch \da v = g &  0<x<1, \hd 0<t<T,\\
(ii) & v_{x}(\cdot,t) \in {}_{0}H^{\al}(0,1),& 0<t<T, \\
(iii) &  v(1,t) = 0, &  0<t<T, \\
(iv) & v(x,0) = v_{0}(x) &  0<x<1, \\
\end{array} \end{equation}
where
\eqq{
v_{0}(x) := w_{0}(x) - h(0)b^{\al}\vf(x)
}{defv0}
and
\eqq{
g(x,t) := -\frac{d}{dt}(h(t)s^{\al}(t))\vf(x) + \dot{s}(t)s^{\al-1}(t)h(t)x\vf'(x) + \frac{h(t)}{s(t)}\poch \da \vf(x).
}{defg}

Then, if  $v$ is a solution to \eqref{niecv} for a given $\varphi$  verifying \eqref{auxvar}, we have that 

\begin{equation}\label{u=reg+sing}
u(x,t) = u_{reg}(x,t) + u_{sin}(x,t) \end{equation}
is a solution to \eqref{niech}, where 
  $ u_{reg}(x,t):= v\left(\frac{x}{s(t)},t\right)$  is the \textsl{regular part of $u$}  and \linebreak
    $ u_{sin}(x,t):= h(t)s^{\al}(t)\varphi\left(\frac{x}{s(t)}\right)$ is the \textsl{singular part of $u$}.

\subsection{Existence of a mild solution to (\ref{niecv}).}
In this subsection we will find a solution to (\ref{niecv}). Let us denote
\eqq{
A(t) = \frac{1}{s^{1+\al}(t)} \poch \da, \hd  D(A(t)) = \dd \m{ for every }t \in [0,T].}{A} 
Then from (\ref{zals}) and Theorem~\ref{analitycznosc} we infer that $A(t)$ generates a family of evolution operators
 $\{G(t,\sigma): 0\leq \sigma\leq t \leq T\}$. 
 Furthermore, we obtain the following simple lemma.
 \begin{lemma}\label{A-estimation} For every $w \in \dd$, there exist constants $c=c(\al,b,M,T)$ and $C=C(\al,b,M,T)$ such that 
\begin{equation}\label{A-dd-rel}
    c\norm{w}_{\dd}\leq \norm{A(t)w}_{\ld}\leq C\norm{w}_{\dd}, \quad \forall \, 0\leq t\leq T.
\end{equation}
 \end{lemma} 
 \begin{proof}
 From Proposition \ref{propo frac} item \ref{propClave} we have that
 $\norm{A(t)w}_{\ld}=\frac{\norm{\frac{\p}{\p x}\da w}_{\ld}}{s(t)^{1+\al}}=\frac{\norm{\p^\al  w_x}_{\ld}}{s(t)^{1+\al}}$.
 Then \eqref{A-dd-rel} holds by applying assumptions \eqref{zals}, Proposition \ref{eq} and Poincaré's inequality (which holds due to the fact that $w(1)=0$).
 \end{proof}

 Our method of the proof relies on the theory of complex interpolation. Thus, we will consider a family of spaces $[\ld,\dd]_{\delta}$ for $\delta \in [0,1]$, where clearly $[\ld,\dd]_{0} = \ld$ and $[\ld,\dd]_{1} = \dd$. Let us recall the properties of evolution operators, which can be found in \cite[Chapter 6]{Lunardi}, that will be used further.
For $f \in \ld$ we have
\eqq{\norm{G(t,\sigma)f}_{\ld} \leq c \norm{f}_{\ld}. }{r0}
If $f \in [L^{2}(0,1),\dd]_{\delta}$ then for any $0\leq\sigma<t\leq T$ and every $\delta \in (0,1)$
\eqq{
\norm{G(t,\sigma)f}_{\dd} \leq \frac{c}{(t-\sigma)^{1-\delta}}\norm{f}_{[L^{2}(0,1),\dd]_{\delta}}.
}{r1}
Moreover, for any $0 \leq \delta < \theta<1$, we have
\eqq{
\norm{G(t,\sigma)f}_{[L^{2}(0,1),\dd]_{\theta}} \leq \frac{c}{(t-\sigma)^{\theta-\delta}}\norm{f}_{[L^{2}(0,1),\dd]_{\delta}}
}{r4}
and for $0\leq \theta<\delta<1$
\eqq{
\norm{A(t)G(t,\sigma)f}_{[L^{2}(0,1),\dd]_{\theta}} \leq \frac{c}{(t-\sigma)^{1+\theta-\delta}}\norm{f}_{[L^{2}(0,1),\dd]_{\delta}}.
}{r2}
Furthermore, for every $0\leq \theta < \delta<1$ and $0\leq \sigma < r < t \leq T$
\eqq{
\norm{G(t,\sigma)f - G(r,\sigma)f}_{[\ld,\dd]_{\theta}} \leq \frac{c}{(t-r)^{\theta-\delta}}\norm{f}_{[\ld,\dd]_{\delta}}.
}{r5}
Finally, for every $a \in (0,1)$ and every $0\leq \sigma < r < t \leq T$
\[
\norm{A(t)G(t,\sigma)f-A(r)G(r,\sigma)f}_{L^{2}(0,1)}
\]
\eqq{
\leq c\left(\frac{(t-r)^{a}}{(r-\sigma)^{1-\delta}}+\frac{1}{(r-\sigma)^{1-\delta}}-\frac{1}{(t-\sigma)^{1-\delta}}\right)
\norm{f}_{[L^{2}(0,1),\dd]_{\delta}}.
}{r3}
The constant $c$ in estimates above is positive and depends only on $\al,\theta,\delta,T$ and $b,M$ from (\ref{zals}). Moreover, function $T\mapsto c(\al,\theta,\delta,b,M,T)$ is increasing.  We also recall the following characterization for the interpolation spaces $[\ld,\dd]_{\delta}$:

\begin{equation}\label{p60}
[\ld,\dd]_{\delta}=\begin{cases}
H^{(1+\al)\delta}(0,1) & \text{if } \delta \in \left(0,\frac{1}{2(1+\al)} \right), \, \, \al \in (0,1),\\
^0H^{1/2}(0,1) & \text{if } \delta=\frac{1}{2(1+\al)} , \, \, \al \in (0,1),\\
\left\{u\in  H^{(1+\al)\delta}(0,1): u(1)=0\right\} & \text{if } \delta \in \left(\frac{1}{2(1+\al)},1 \right), \, \al \in \left(0,\frac{1}{2}\right],\\
\left\{u\in  H^{(1+\al)\delta}(0,1): u(1)=0\right\} & \text{if } \delta \in \left(\frac{1}{2(1+\al)},\frac{3}{2(1+\al)} \right), \al \in \left(\frac{1}{2},1\right),\\
\left\{u\in  H^{3/2}(0,1): u(1)=0,\, \, u_x \in\, _0H^{1/2}(0,1)\right\} & \text{if } \delta =\frac{3}{2(1+\al)}, \, \, \al \in \left(\frac{1}{2},1\right),\\
\left\{u\in  H^{(1+\al)\delta}(0,1): u(1)=0,\, \, u_x(0)=0\right\} & \text{if } \delta \in \left(\frac{3}{2(1+\al)},1\right), \, \, \al \in \left(\frac{1}{2},1\right).
\end{cases}
\end{equation}
And the useful property by Lions and Magenes \cite[Remark 12.8.]{LiMaV1} which states that for every $s\neq \frac{1}{2}$  there holds 
\begin{equation} \label{Lions}
\partial_x \in  \mathcal{L}(H^s(0,1); H^{s-1}(0,1)). \end{equation}

Now, our next step is to establish the regularity of $g$   which appears at the right-hand-side of $(\ref{niecv}-{i})$.
We note that  we cannot guarantee that $g(1,t) = 0$, thus, according to the characterization \eqref{p60}, $g$ does not belongs to the interpolation space $[\ld,\dd]_{\delta}$ for $\delta > \frac{1}{2(1+\al)}$. Anyway, since $\poch \da x^{\al} = 0$ and $\vf - \frac{x^{\al}}{\Gamma(1+\al)}$ is smooth, assuming that  $h \in W^{1,\infty}(0,T)$, we obtain the following.

\begin{prop}\label{regularity-g} 
The function $g$ defined in \eqref{defg} verifies that 
\eqq{
g \in L^{\infty}(0,T;[\ld,\dd]_{\delta}) \m{ for every } \delta \in \left(0,\frac{1}{2(1+\al)}\right)
}{regg}
and there exists positive $c=c(b,M,T,\vf)$ such that
\eqq{
\norm{g}_{L^{\infty}(0,T;[\ld,\dd]_{\delta})} \leq c \norm{h}_{W^{1,\infty}(0,1)}.
}{estg}
Furthermore, since $x^{\al} \in {}_{0}H^{\beta}(0,1)$ for any $\beta < \al+\frac{1}{2}$ we obtain 
\eqq{
g \in L^{\infty}(0,T;{}_{0}H^{\beta}(0,1)) \m{ for any } \beta < \al+\frac{1}{2}.
}{reggh}
and for any $0<\ve<1$ there holds $g \in L^{\infty}(0,T;H^{1+\al}(\ve,1))$.
\end{prop}


Let us  search for a mild solution to (\ref{niecv}) in the following functional space defined for every $\gamma \in (0,1)$ as 
\[
Y_{\gamma}(0,T):=\{u\in C([0,T];\ld): t^{\gamma}u \in C([0,T];\dd)\}
\]
with the norm
\[
\norm{u}_{Y_{\gamma}(0,T)} = \norm{u}_{C([0,T];\ld)} + \sup_{t\in [0,T]}t^{\gamma}\norm{u(\cdot,t)}_{\dd}
\]
\begin{defi}\label{mild-sol} We say that $v \in  Y_{\gamma}(0,T)$  is a mild solution to $(\ref{niecv})$ if it satisfies
\eqq{
v(x,t) =  G(t,0)v_{0}(x) + \int_{0}^{t}G(t,\sigma)\frac{\dot{s}(\sigma)}{s(\sigma)}xv_{x}(x,\sigma)\rmd \sigma + \int_{0}^{t}G(t,\sigma)g(x,\sigma)\rmd \sigma .
 }{wzorv}
 \end{defi}

\begin{note} The existence and regularity results of the solution to (\ref{niecv}), given below, are the adaptation of analogous results given in \cite{Rys:2020}. However, it is important to note that here we consider less regularity of the initial data and a non-zero source term.
 \end{note}

We note that if $v \in \dd$ we may not expect that $v_{x}$ vanishes at the right endpoint of the interval. In this case, we obtain that $v_{x}(\cdot,t) \in [L^{2}(0,1),\dd]_{\delta}$ only for $\delta < \frac{1}{2(1+\al)}$. We will use this fact in the proof of the result concerning existence of mild solutions, which is given below.
\begin{theo}\label{first}
Let us assume that $h \in W^{1,\infty}(0,T)$, $v_{0} \in H^{\theta(\al+1)}(0,1)$ for fixed $\theta \in (0,\frac{1}{2(1+\al)})$.
Then, for every $1-\theta<\gamma < 1$ there exists a unique mild solution to~(\ref{niecv}) belonging to $Y_{\gamma}(0,T)$.
\end{theo}
\begin{proof}
In the proof we will apply the Banach fixed point theorem. To this end, we define operator $P$ by
\[
(Pv)(x,t)=  G(t,0)v_{0}(x) + \int_{0}^{t}G(t,\sigma)\frac{\dot{s}(\sigma)}{s(\sigma)}xv_{x}(x,\sigma)\rmd\sigma + \int_{0}^{t}G(t,\sigma)g(x,\sigma)\rmd \sigma.
\]
We will show that for $1-\theta<\gamma < 1$ there holds $P:Y_{\gamma}(0,T) \rightarrow Y_{\gamma}(0,T)$.
At first, we notice that by the property of evolution operators $G(t,0)v_{0} \in C([0,T];\ld)$ and taking into account the regularity of $v_0$ together with the characterization \eqref{p60}, Lemma \ref{A-estimation} and \eqref{r3} we deduce that $G(t,0)v_{0} \in C((0,T];\dd)$. Moreover, by (\ref{r1}) we have
\[
\norm{t^{\gamma}G(t,0)v_{0}}_{\dd} \leq ct^{\gamma+\theta-1}\norm{v_{0}}_{[\ld,\dd]_{\theta}}\rightarrow 0 \m{ as } t\rightarrow 0 \quad \text{ if }\, \gamma>1-\theta.
\]

\no Hence, $G(t,0)v_{0} \in Y_{\gamma}(0,T)$.
Now, let us denote
\eqq{
f(x,\sigma):=\frac{\dot{s}(\sigma)}{s(\sigma)}xv_{x}(x,\sigma).
}{deff}
If $v \in Y_{\gamma}(0,T)$, we obtain that
\eqq{
t^{\gamma}f \in L^{\infty}((0,T);{}_{0}H^{\al}(0,1))
}{vxm-alpha}
and hence, recalling \eqref{p60} and (\ref{regg})  we obtain that
\eqq{
t^{\gamma}[f+g] \in L^{\infty}((0,T);[\ld,\dd]_{\delta}),
}{vxm}
where $\delta = \frac{\al}{\al+1}$ in the case $\al \in (0,\frac{1}{2})$ and in the case $\al \in [\frac{1}{2},1)$ by $\delta$ we mean an arbitrary number from the interval $(0,\frac{1}{2(\al+1)})$.
For any $0\leq \tau<t \leq T$ we may estimate as follows
\[
\norm{\izt G(t,\sigma)[f+g](\cdot,\sigma)\rmd \sigma -\int_{0}^{\tau} G(\tau,\sigma)[f+g](\cdot,\sigma)\rmd \sigma}_{L^{2}(0,1)}
\]
\[
\leq \int_{\tau}^{t}\norm{G(t,\sigma)[f+g](\cdot,\sigma)}_{L^{2}(0,1)}\rmd \sigma + \int_{0}^{\tau}\hspace{-0.2 cm}\norm{(G(t,\sigma)-G(\tau,\sigma))[f+g](\cdot,\sigma)}_{L^{2}(0,1)}\rmd \sigma
\]
\[
\leq c \norm{t^{\gamma}[f+g]}_{L^{\infty}(0,T;\ld)}  \int_{\tau}^{t}\sigma^{-\gamma} \rmd\sigma + c \norm{t^{\gamma}[f+g]}_{L^{\infty}(0,T;[\ld,\dd]_{\delta})}  \int^{\tau}_{0}\sigma^{-\gamma}(t-\tau)^{\delta}\rmd\sigma,
\]
where in the last estimate we applied (\ref{r0}) to the first term and (\ref{r5}) to the second term. The expression above tends to zero
as $\tau \rightarrow t$ for any $0\leq \tau<t \leq T$, thus $\izt G(t,\sigma)[f+g](\cdot,\sigma)\rmd\sigma \in C([0,T],\ld)$. In view of Lemma \ref{A-estimation} it remains to show that $t^{\gamma}A(t)\izt G(t,\sigma)[f+g](\cdot,\sigma)\rmd\sigma \in C([0,T];\ld)$. Concerning continuity away form zero, since $t^\gamma$ is a continuous function, it is enough to show  that $A(t)\izt G(t,\sigma)[f+g](\cdot,\sigma)\rmd\sigma \in C((0,T];\ld)$.
For any $0< \tau<t \leq T$ we may estimate as follows
\[
B:=\norm{A(t)\izt G(t,\sigma)[f+g](\cdot,\sigma)\rmd\sigma -A(\tau)\int_{0}^{\tau} G(\tau,\sigma)[f+g](\cdot,\sigma)\rmd\sigma}_{L^{2}(0,1)}
\]

\[\leq
\int_{\tau}^{t}\norm{A(t)G(t,\sigma)[f+g](\cdot,\sigma)}_{L^{2}(0,1)}\rmd\sigma + \int_{0}^{\tau}\norm{(A(t)G(t,\sigma)-A(\tau)G(\tau,\sigma))[f+g](\cdot,\sigma)}_{L^{2}(0,1)}\rmd\sigma
\]
By (\ref{r2}) and (\ref{r3}) we may estimate  further,

\[
B \leq c \norm{t^{\gamma}[f+g]}_{L^{\infty}(0,T;[L^{2}(0,1),\dd]_{\delta})}\int_{\tau}^{t} \sigma^{-\gamma}(t-\sigma)^{\delta-1}\rmd\sigma
\]
\[
 + c \norm{t^{\gamma}[f+g]}_{L^{\infty}(0,T;[L^{2}(0,1),\dd]_{\delta})}\int_{0}^{\tau}\sigma^{-\gamma}\left[\frac{(t-\tau)^{a}}{(\tau-\sigma)^{1-\delta}} + \frac{1}{(\tau-\sigma)^{1-\delta}}-\frac{1}{(t-\sigma)^{1-\delta}}\right] \rmd\sigma
\]
for any $a \in (0,1)$ and the last expression tends to zero as $\tau\rightarrow t$ for any $0<\tau<t \leq T$. To show the continuity up to zero we notice that, by (\ref{r1}) we may estimate as follows
\[
\norm{t^{\gamma}\izt G(t,\sigma)[f+g](\cdot,\sigma)\rmd\sigma}_{\dd} \leq t^{\gamma}\izt \norm{G(t,\sigma)[f+g](\cdot,\sigma)}_{\dd}\rmd\sigma
\]
\[
\leq c t^{\gamma}\izt (t-\sigma)^{\delta-1}\norm{[f+g](\cdot,\sigma)}_{[\ld,\dd]_{\delta}}\rmd\sigma
\]
\[
\leq c \norm{t^{\gamma}[f+g]}_{L^{\infty}(0,T;[L^{2}(0,1),\dd]_{\delta})}t^{\gamma}\izt \sigma^{-\gamma} (t-\sigma)^{\delta-1}\rmd\sigma \rightarrow 0 \m{ as } t\rightarrow 0.
\]
Hence we proved that $\int_{0}^{t}G(t,\sigma)[f+g](x,\sigma)\rmd\sigma \in Y_{\gamma}(0,T)$ and thus
$P:Y_{\gamma}(0,T)\rightarrow Y_{\gamma}(0,T)$.\\
\\
\no Now, we will show that $P$ is a contraction on $Y_{\gamma}(0,T_{1})$ for $T_{1}$ small enough. We note that, if $ t^\gamma v \in  C([0, T]; \dd)$, then  $ t^\gamma v_x \in  C([0, T ]; \, _0H^\al(0,1))$ and  $t^\gamma xv_x \in C([0,T];\, _0H^\al(0, 1))$.

\no Then, if we take $v_{1},v_{2} \in Y_{\gamma}(0,T_{1})$, 
\[
P(v_{1})(x,t) - P(v_{2})(x,t)= \izt G(t,\sigma)\frac{\dot{s}(\sigma)}{s(\sigma)}x(v_{1,x} - v_{2,x})(x,\sigma)\rmd\sigma
\]
and applying (\ref{r0}) we obtain that
\[
\norm{P(v_{1}) - P(v_{2})}_{C([0,T_{1}],\ld)} \leq \sup_{t \in [0,T_{1}]}\izt\norm{G(t,\sigma)\frac{\dot{s}(\sigma)}{s(\sigma)}x(v_{1,x} - v_{2,x})(\cdot,\sigma)}_{\ld}\rmd\sigma
\]
\[
\leq  c \sup_{t \in [0,T_{1}]}\norm{\frac{\dot{s}(t)}{s(t)}t^{\gamma}x(v_{1,x} - v_{2,x})(\cdot,t)}_{\ld}\izt \sigma^{-\gamma}\rmd\sigma
\leq c\frac{M}{b}  \norm{v_{1}-v_{2}}_{Y_{\gamma}(0,T_{1})}\frac{T_{1}^{1-\gamma}}{1-\gamma},
\]
where the positive constant $c$ comes from (\ref{r0}).
Hence, there exists $\lambda \in (0,1)$ such that
\[
\norm{P(v_{1})(t) - P(v_{2})}_{C([0,T_{1}],\ld)} \leq \lambda \norm{v_{1}-v_{2}}_{Y_{\gamma}(0,T_{1})},
\]
whenever $T_{1} < \left(\frac{b(1-\gamma)}{c\bar{c}M}\right)^{\frac{1}{1-\gamma}}$.
Furthermore, applying (\ref{r1}) we have
\[
\sup_{t \in [0,T_{1}]}t^{\gamma}\norm{P(v_{1}) - P(v_{2})}_{\dd} \leq \sup_{t \in [0,T_{1}]}t^{\gamma}\izt\norm{G(t,\sigma)\frac{\dot{s}(\sigma)}{s(\sigma)}x(v_{1,x} - v_{2,x})(\cdot,\sigma)}_{\dd}d\sigma
\]
\[
\leq \sup_{t \in [0,T_{1}]}t^{\gamma}\norm{\frac{\dot{s}(t)}{s(t)}t^{\gamma}x(v_{1,x} - v_{2,x})(\cdot,t)}_{[\ld,\dd]_{\delta}}c\izt \sigma^{-\gamma}(t-\sigma)^{\delta-1}d\sigma
\]
\[
\leq c \frac{M}{b} \norm{v_{1}-v_{2}}_{Y_{\gamma}(0,T_{1})}\sup_{t \in [0,T_{1}]}t^{\gamma}\izt \sigma^{-\gamma}(t-\sigma)^{\delta-1}d\sigma
\leq c \frac{M}{b} \norm{v_{1}-v_{2}}_{Y_{\gamma}(0,T_{1})} T_{1}^{\delta}\frac{\Gamma(\delta)\Gamma(1-\gamma)}{\Gamma(1+\delta-\gamma)},
\]
where the constant $c>0$ comes from (\ref{r1}).
Thus, there exists $\lambda \in (0,1)$ such that
\[
\sup_{t \in [0,T_{1}]}t^\gamma\norm{P(v_{1})(t) - P(v_{2})}_{\dd} \leq \lambda \norm{v_{1}-v_{2}}_{Y_{\gamma}(0,T_{1})},
\]
whenever
\[
T_{1} < \left(\frac{b}{cM}\frac{\Gamma(1+\delta-\gamma)}{\Gamma(\delta)\Gamma(1-\gamma)}\right)^{\frac{1}{\delta}}.
\]
Finally, by the Banach fixed point theorem we obtain that there exists a unique solution to (\ref{wzorv}) belonging to $Y_{\gamma}(0,T_{1})$ for any
$T_{1} < \min\left\{\left(\frac{b(1-\gamma)}{cM}\right)^{\frac{1}{1-\gamma}},\left(\frac{b}{cM}\frac{\Gamma(1+\delta-\gamma)}{\Gamma(\delta)\Gamma(1-\gamma)}\right)^{\frac{1}{\delta}}\right\}.$
We may extend the solution on the whole interval $[0,T]$ applying the standard argument.
\end{proof}

 \begin{lemma}\label{second}
The mild solution $v$ obtained in Theorem \ref{first} satisfies additionally $v_{t} \in L^{\infty}_{loc}((0,T];\ld)$ and
  \[
  v_{t} - x\frac{\dot{s}(t)}{s(t)}v_{x} - \frac{1}{s^{1+\al}}\poch \da v = g
  \]
  in  $L^{2}(0,1)$, for almost all $t\in (0,T)$. 
  \end{lemma}
	
\begin{proof}
  Using definition (\ref{deff}) we may rewrite (\ref{wzorv}) as follows
  \eqq{
  v(x,t) = G(t,0)v_{0}(x) + \izt G(t,\sigma)f(x,\sigma)\rmd\sigma+ \izt G(t,\sigma)g(x,\sigma)\rmd\sigma.
  }{vcf}
  From the properties of evolution operators we infer that for every $t \in (0,T]$ we have
  \[
  \frac{\p}{\p t}G(t,0)v_{0} = A(t) G(t,0)v_{0} \m{ in } L^{2}(0,1).
  \]
  We will calculate the difference quotient of the second term on the right hand side of~(\ref{vcf}). Let us assume that $h > 0$, in the case $h<0$ the proof is similar.
  We have
  \[
  \frac{1}{h}\left[\int_{0}^{t+h} G(t+h,\sigma)f(x,\sigma)\rmd\sigma - \izt G(t,\sigma)f(x,\sigma)\rmd\sigma\right]
  \]
  \[
  = \frac{1}{h}\int_{0}^{t} (G(t+h,\sigma)-G(t,\sigma))f(x,\sigma)\rmd\sigma
  +\frac{1}{h}\int_{t}^{t+h} G(t+h,\sigma)f(x,\sigma)\rmd\sigma =: I_{1} + I_{2}.
  \]
  In order to deal with $I_{1}$ we recall that for every $0 \leq \sigma < t \leq T$ and every $z \in L^{2}(0,1)$ the following limit holds in $L^{2}(0,1)$
  \[
  \lim_{h\rightarrow 0}\frac{1}{h}(G(t+h,\sigma)-G(t,\sigma))z = A(t)G(t,\sigma)z.
  \]
    Making use of (\ref{vxm}) we obtain that $t^{\gamma}f \in L^{\infty}(0,T;[L^{2}(0,1),\dd]_{\delta})$, where $\delta = \frac{\al}{1+\al}$ for $\al \in (0,\frac{1}{2})$ and $\delta$ denotes any fixed number from the interval $(0,\frac{1}{2(1+\al)})$ if $\al \in [\frac{1}{2},1)$. Further, applying \eqref{r2} we obtain that  
  \[
  \norm{\frac{1}{h}[G(t+h,\sigma) - G(t,\sigma)]f(\cdot,\sigma)}_{L^{2}(0,1)} =
  \norm{\frac{1}{h}\int_{t}^{t+h}\frac{\p}{\p p}G(p,\sigma)f(\cdot,\sigma)\rmd p}_{L^{2}(0,1)}
  \]
  \[
  =\norm{\frac{1}{h}\int_{t}^{t+h}A(p)G(p,\sigma)f(\cdot,\sigma)\rmd p}_{L^{2}(0,1)}
  \leq \frac{c}{h}\int_{t}^{t+h}(p-\sigma)^{\delta-1}\rmd p\,\sigma^{-\gamma}\norm{t^{\gamma}f}_{L^{\infty}(0,T;[L^{2}(0,1),\dd]_{\delta})}
  \]
  \[
  \leq c(t-\sigma)^{\delta-1}\sigma^{-\gamma}\norm{t^{\gamma}f}_{L^{\infty}(0,T;[L^{2}(0,1),\dd]_{\delta})}.
  \]

 Hence, we may apply the Lebesgue dominated convergence theorem to pass to the limit under the integral sign in $I_{1}$ and we get
  \[
  \frac{1}{h}\int_{0}^{t} (G(t+h,\sigma)-G(t,\sigma))f(x,\sigma)d\sigma\rightarrow \int_{0}^{t} A(t)G(t,\sigma)f(x,\sigma)\rmd\sigma.
  \]
  We decompose $I_{2}$ as follows
  \[
  \frac{1}{h}\int_{t}^{t+h} G(t+h,\sigma)f(x,\sigma)\rmd\sigma= \frac{1}{h}\int_{t}^{t+h} G(t,\sigma)f(x,\sigma)\rmd\sigma
  \]\[
   +\frac{1}{h}\int_{t}^{t+h} (G(t+h,\sigma)-G(t,\sigma))f(x,\sigma)\rmd\sigma = I_{2,1} + I_{2,2}.
  \]
We note that due to the Lebesgue differentiation theorem  in Banach spaces we obtain that $I_{2,1}$ converges to $f(x,t)$ in $L^{2}(0,1)$ for almost all $t \in (0,T)$. For $I_{2,2}$, denoting by $E$ the identity, we have
\[
I_{2,2} = \frac{1}{h}\int_{t}^{t+h} (G(t+h,t)-E)G(t,\sigma)f(x,\sigma)\rmd\sigma =  \frac{(G(t+h,t)-E)}{h}\int_{t}^{t+h}G(t,\sigma)f(x,\sigma)\rmd\sigma.
\]
Thus, using again the Lebesgue differentiation theorem in Banach spaces and the continuity of $G(t,\cdot)$ in $L^{2}(0,1)$ we obtain that $I_{2,2}$ converges to zero in $L^{2}(0,1)$ for almost all $t \in (0,T)$.
Summing up the results we obtain that the following identity holds in $L^{2}(0,1)$ for almost all $t \in [0,T]$
 \[
  \lim_{h\rightarrow 0^{+}}\frac{1}{h}\left[\int_{0}^{t+h} G(t+h,\sigma)f(x,\sigma)\rmd\sigma - \izt G(t,\sigma)f(x,\sigma)\rmd\sigma\right]
  \]
   \[
   = A(t)G(t,0)v_{0}(x) + A(t)\izt G(t,\sigma)f(x,\sigma)\rmd\sigma + f(x,t).
  \]
Proceeding similarly we may obtain analogous identity for $g$ and hence applying the formula (\ref{vcf}) and the definitions of $f$ and $A$ we arrive at
\[
v_{t}(x,t) = \frac{1}{s^{1+\al}(t)}\poch \da v(x,t) + \frac{\dot{s}(t)}{s(t)}xv_{x}(x,t) + g(x,t)
\]
in $L^{2}(0,1)$, for almost all $t \in (0,T)$  and we obtain the claim of lemma. The regularity comes from \eqref{vxm} and \eqref{zals}.
  \end{proof}	
\subsection{Improved  regularity of the mild solution.}
Our goal now is to improve the regularity of the mild solution to \eqref{niecv}, given by Theorem \ref{first}. The proof of the existence of the solution to Stefan problem relies on the maximum principle. Because of that, here we would like to obtain that the solution to (\ref{niecv}) satisfies (\ref{regumax}).

As we have already noted, from the definition of $\dd$ we may not expect that $v_{x}$ vanishes at the right endpoint of the interval. Hence, even having $v$ smooth, we obtain that $v_{x}(\cdot,t) \in [L^{2}(0,1),\dd]_{\delta}$ only for $\delta < \frac{1}{2(1+\al)}$. That is why in a proof of increased regularity of the solution to (\ref{niecv}) we will make use of the following results.  

\begin{lemma}\label{lemusi}\cite[Lemma 2]{Rys:2020}
Let us assume that $\al \in (\frac{1}{2},1)$ and $u_{\sigma} \in {}_{0}H^{\al}(0,1)$. We denote by $u$ the solution to the problem
\begin{equation}\label{pol}
 \left\{ \begin{array}{ll}
u_{t} =A(t)u & \textrm{ for }\hd 0<x<1, \hd 0\leq \sigma<t<T,\\
u(x,\sigma) = u_{\sigma}(x) & \textrm{ for }\hd 0<x<1, \\
\end{array} \right. \end{equation}
given by the evolution operator generated by the family $A(t)$. Then, for every \hd $0<\gamma<\al$, for every $0<\ve<\omega<1$ there exists a positive constant $c=c(\al,b,M,T,\ve,\omega,\gamma)$, where $b,M$ comes from (\ref{zals}), such that 
\[
\norm{A(t)u(\cdot,t)}_{H^{\gamma}(\ve,\omega)} \leq c(t-\sigma)^{-\frac{1+\gamma}{1+\al}}\norm{u_{\sigma}}_{{}_{0}H^{\al}(0,1)}, \quad \text{for every} \, \, t \in (\sigma,T].
\]
\end{lemma}

\begin{lemma}\label{lemusi2}\cite[Lemma 3]{Rys:2020}
Let $0<\al \leq \frac{1}{2}$. Let us assume that $u_{\sigma} \in H^{\beta}_{loc}(0,1) \cap H^{\bar{\gamma}}(0,1)$, where $\frac{1}{2}<\beta<1$ and $0<\bar{\gamma} < \frac{1}{2}$ are fixed. We denote by $u$ the solution to the problem
\begin{equation}\label{polc}
 \left\{ \begin{array}{ll}
u_{t} =A(t)u & \textrm{ for } 0<x<1, \hd 0\leq \sigma<t<T,\\
u(x,\sigma) = u_{\sigma}(x) & \textrm{ for } 0<x<1, \\
\end{array} \right. \end{equation}
given by the evolution operator generated by the family $A(t)$. Then, for every  $\max\{\beta-\al, \beta - \bar{\gamma}\}<\beta_{1}<\beta$, for every $0<\ve<\omega<1$, there exists a positive constant $c=c(\al,b,M,T,\ve,\omega,\beta,\beta_{1})$, such that for every $t \in (\sigma,T]$ there holds
\[
\norm{A(t)u(\cdot,t)}_{H^{\beta_{1}}(\ve,\omega)} \leq c(t-\sigma)^{-\frac{\beta_{1}-\beta+\al+1}{1+\al}}(\norm{u_{\sigma}}_{H^{\beta}(\frac{\ve}{2},\frac{1+\omega}{2})}+\norm{u_{\sigma}}_{H^{\bar{\gamma}}(0,1)}).
\]
\end{lemma}

\begin{lemma}\label{third}
 Let us assume that $h \in W^{1,\infty}(0,T)$, $v_{0} \in {}^{0}H^{\theta(\al+1)}(0,1)$ for fixed $\theta \in (\frac{1}{2(1+\al)},1]$, and that  $v$ is the solution to (\ref{niecv}) obtained in Theorem~\ref{first}. Then the following results hold:

\begin{enumerate}

\item For every $\al\in (0,1)$ and every $\bar{\gamma}\in \left(0,\frac{1}{2}\right)$ there holds
\begin{equation}\label{vxb}
v_{x} \in L^{\infty}_{loc}(0,T;{_0}H^{\al+\bar{\gamma}}(0,1)).
\end{equation}

\item There exists $\hat{\gamma}>\frac{1}{2(1+\al)}$ such that 
\begin{equation}\label{reggammahat} v \in L^{\infty}(0,T;[\ld,\dd]_{\hat{\gamma}}).\end{equation}
\item There exists $\beta>\frac{1}{2}$ such that
\begin{equation}\label{regvbeta}
v \in L^{\infty}_{loc}(0,T;H^{\beta+1+\al}_{loc}(0,1)) \m{ and } \p^{\al}v_{x} \in L^{\infty}_{loc}(0,T;H^{\beta}_{loc}(0,1)).
\end{equation}

\item Moreover, there exists $\beta>\frac{1}{2}$ such that
\begin{equation}\label{regv-cont}
v \in C((0,T];H^{\beta+1+\al}_{loc}(0,1)) \m{ and } \p^{\al}v_{x} \in C((0,T];H^{\beta}_{loc}(0,1)).
\end{equation}

\end{enumerate}

\end{lemma}

\begin{proof}
We divide the proof into a few cases. Before start, we would notice that the constant $\gamma$ given in Theorem \ref{first} can be an arbitrary number in the interval $\left(1-\frac{1}{2(\al+1)}, 1\right)$ due to the fact that  the regularity assumed over the initial condition in this Lemma is stronger than the one required Theorem \ref{first}.

\vspace{0.5cm}

{\bf Case 1: $\al\in \left(\frac{1}{2},1\right)$}

\begin{enumerate}

\item Let $\bar{\gamma}\in \left(0,\frac{1}{2}\right)$. Then, $\frac{\bar{\gamma}}{1+\al}\in \left(0,\frac{1}{2(1+\al)}\right)$ and from \eqref{vxm} $t^\gamma (f+g)\in L^{\infty}(0,T;[L^2(0,1),\dd]_{\frac{\bar{\gamma}}{1+\al}})$. 

We apply the operator $A(t)$ to (\ref{vcf}) and estimate its norm in an interpolation space to obtain that
\begin{equation}
\begin{array}{l}
\norm{A(t)v}_{[L^2(0,1),\dd]_{\frac{\bar{\gamma}}{1+\al}}} \leq\\
\leq \norm{A(t)G(t,0)v_{0}}_{[L^2(0,1),\dd]_{\frac{\bar{\gamma}}{1+\al}}} + \norm{A(t)\izt G(t,\sigma)[f(\cdot,\sigma)+g(\cdot,\sigma)]d\sigma}_{[L^2(0,1),\dd]_{\frac{\bar{\gamma}}{1+\al}}}\\
\leq \norm{A(t)G(t,0)v_{0}}_{[L^2(0,1),\dd]_{\frac{\bar{\gamma}}{1+\al}}} + \izt \sigma^{-\gamma}\norm{A(t)G(t,\sigma)\sigma^{\gamma}[f(\cdot,\sigma)+g(\cdot,\sigma)]}_{[L^2(0,1),\dd]_{\frac{\bar{\gamma}}{1+\al}}}\rmd\sigma\\
\end{array}
\end{equation}
\normalsize
Now we consider $0<\bar{\gamma}<\mu<\frac{1}{2}$ and from \eqref{r2} we obtain that
\begin{equation}
\begin{array}{l}
\norm{A(t)v}_{[L^2(0,1),\dd]_{\frac{\bar{\gamma}}{1+\al}}} \leq\\
\leq c_1 t^{-1-\frac{\bar{\gamma}}{1+\al}+\theta}\hspace{-0.1cm}\norm{v_{0}}_{[L^2(0,1),\dd]_{\theta}}\hspace{-0.1cm} + \hspace{-0.1cm}c_2 \izt(t-\sigma)^{-1-\frac{\bar{\gamma}}{1+\al}+\frac{\mu}{1+\al}}\sigma^{-\gamma}\rmd \sigma \norm{t^{\gamma}(f+g)}_{L^{\infty}(0,T;[L^2(0,1),\dd]_{\frac{\mu}{1+\al}})}\\
= c_1 t^{-1-\frac{\bar{\gamma}}{1+\al}+\theta}\norm{v_{0}}_{[L^2(0,1),\dd]_{\theta}}+ c_2 \left(\frac{\mu-\bar{\gamma}}{1+\al},1-\gamma\right) t^{\frac{\mu-\bar{\gamma}}{1+\al}-\gamma}\norm{t^{\gamma}(f+g)}_{L^{\infty}(0,T;[L^2(0,1),\dd]_{\frac{\mu}{1+\al}})}.
\end{array}
\end{equation}
\normalsize

Thus, $A(t)v \in L^{\infty}_{loc}(0,T;[L^2(0,1),\dd]_{\frac{\bar{\gamma}}{1+\al}})$ for every $\bar{\gamma} \in (0,\frac{1}{2})$ which in view of (\ref{zals}) and Corollary~\ref{eqc} leads to (\ref{vxb}) in the case $\al \in  (\frac{1}{2},1)$ due to the fact that  $[L^2(0,1),\dd]_{\frac{\bar{\gamma}}{1+\al}}={_0}H^{\bar{\gamma}}(0,1) $  when $\bar{\gamma}\in \left(0,\frac 1 2\right)$.

\item Let $\gamma$ be the constant given in Theorem \ref{first}. Then $1-\gamma>0$ and we can set $k\in \bbN$ such that $\hat{\gamma}:=\frac{1}{2(1+\al)}+\frac{1-\gamma}{k}$ verifies that 
\begin{equation}\label{gammahat}  \hat{\gamma}<\theta,\quad \quad \hat{\gamma}<\frac{3}{2(1+\al)}.\end{equation}
Form the assumption on $v_0$ and \eqref{p60} it holds that $v_0 \in \,^0H^{\theta(1+\al)}(0,1)\subset \,^0H^{\hat{\gamma}(1+\al)}(0,1)=[L^2(0,1),\dd]_{\hat{\gamma}} $ and we can apply (\ref{r4}) to (\ref{wzorv}) to obtain that for every $\delta<\frac{1}{2(1+\al)}<\hat{\gamma}$
\begin{equation}
\begin{split}
\norm{v(\cdot,t)}_{[\ld,\dd]_{\hat{\gamma}}} \leq \norm{G(t,\sigma)v_{0}}_{[\ld,\dd]_{\hat{\gamma}}}
+ \izt \norm{G(t,\sigma)[f(\cdot,\sigma) + g(\cdot,\sigma)]}_{[\ld,\dd]_{\hat{\gamma}}}d\sigma
\\
\leq c_1\norm{v_{0}}_{[\ld,\dd]_{\hat{\gamma}}} + c_2 \norm{t^{\gamma}[f + g]}_{L^{\infty}(0,T);[\ld,\dd]_{\delta}}\izt (t-\sigma)^{\delta - \hat{\gamma}}\sigma^{-\gamma}d\sigma
\\
\leq c_1\norm{v_{0}}_{[\ld,\dd]_{\hat{\gamma}}} + c_2 B(1+\delta-\hat{\gamma}, 1-\gamma) t^{\delta-\hat{\gamma}+1-\gamma} \norm{t^{\gamma}[f + g]}_{L^{\infty}(0,T);[\ld,\dd]_{\delta}}.
\end{split}
\end{equation}
By choosing a particular $\delta:=\frac{1}{2(1+\al)}-\frac{1-\gamma}{k+1}$ it holds that  $1-\gamma-(\hat{\gamma}-\delta) >0$ and \eqref{reggammahat} holds.


\item Let $\delta=\frac{\al}{1+\al}>\frac{1}{2(1+\al)}$. Then $\delta (1+\al)>1/2$ and we can not work in the interpolation spaces. However we can apply Lemma \ref{lemusi} to obtain that for any $0<\ve<\omega<1$ and any $0 < \bar{\theta} < \delta$

\begin{equation}\label{desinorm1}
\norm{A(t)\izt G(t,\sigma)[f(\cdot,\sigma)+g(\cdot,\sigma)]d\sigma}_{H^{(1+\al)\overline{\theta}}(\ve,\omega)}\hspace{-0.3cm} \leq c B(1-\gamma, \delta - \overline{\theta}) t^{\delta-\overline{\theta}-\gamma}\norm{t^{\gamma}(f+g)}_{L^{\infty}(0,T;{}_{0}H^{\al}(0,1))}.
\end{equation}


Furthermore by (\ref{r2}) we have

\begin{equation}\label{desinorm2}
\norm{A(t)G(t,0)v_{0}}_{H^{(1+\al)\overline{\theta}}(\ve,\omega)}
\leq\norm{A(t)G(t,0)v_{0}}_{[\ld,\dd]_{\overline{\theta}}} \leq  ct^{-1-\overline{\theta}+\theta}\norm{v_{0}}_{[\ld,\dd]_{\theta}}.
\end{equation}

At last,  \eqref{desinorm1} and \eqref{desinorm2} yields that $\norm{A(t)v(\cdot,t)}_{H^{(1+\al)\overline{\theta}}(\ve,\omega)}\leq C(t,\al,\overline{\theta},\delta,h,v_0)$, $\forall t\in(0,T]$. Since $\al>1/2$, we may choose $\overline{\theta}$ large enough so that $\beta:=(1+\al)\overline{\theta}>1/2$. Then $A(t)v\in L^{\infty}_{loc}(0,T;H^{\beta}_{loc}(0,1))$, and from (\ref{zals}) and Lemma \ref{local} \eqref{regvbeta} holds.
\end{enumerate}


\vspace{0.5cm}

{\bf Case 2: $\al\in \left(\frac{1}{4},\frac{1}{2}\right]$}
\begin{enumerate}
\item

Let $\delta=\frac{\al}{1+\al}<\frac{1}{2(1+\al)}$ in the case $\al < \frac{1}{2}$ and for $\al = \frac{1}{2}$ by $\delta$ we mean arbitrary number from the interval $(0,\frac{1}{3})$. Then 
$_{0}H^{\al}(0,1) \subset [L^{2}(0,1),\dd]_{\delta}$ (with equality when $\al < \frac{1}{2}$). From Theorem \ref{first}  $t^\gamma f\in L^\infty(0,T;\, _{0}H^{\al}(0,1))$, thus applying  (\ref{r2}), we obtain that for any $\overline{\theta} < \delta$
\begin{equation}\label{acotAt01}
\begin{array}{l}
\hspace{-0.2cm}\norm{A(t)\izt G(t,\sigma)[f(\cdot,\sigma)+g(\cdot,\sigma)]\rmd \sigma}_{[L^{2}(0,1),\dd]_{\overline{\theta}}} \hspace{-0.1cm}\leq  \izt \frac{c}{(t-\sigma)^{1+\overline{\theta}-\delta}}\norm{t^{\gamma}[f(\cdot,\sigma)+g(\cdot,\sigma)]}_{[L^{2}(0,1),\dd]_{\delta}}d\sigma\\
\leq c B(\delta-\overline{\theta}, 1-\gamma) t^{\delta-\overline{\theta}-\gamma}\norm{t^{\gamma}(f+g)}_{L^{\infty}(0,T;[L^{2}(0,1),\dd]_{\delta})}.
\end{array}
\end{equation}

Also, by (\ref{r2}) we have
\begin{equation}\label{acotAt03}
\norm{A(t)G(t,0)v_{0}}_{[L^{2}(0,1), \dd]_{\overline{\theta}}} \leq  ct^{-1-\overline{\theta}+\theta}\norm{v_{0}}_{[L^{2}(0,1), \dd]_{\theta}}.
\end{equation}
Then, taking into account \eqref{acotAt01} and \eqref{acotAt03}, we deduce that
\begin{equation}\label{acotvx01}
\norm{t^\gamma\partial^\al v_x}_{[L^{2}(0,1),\dd]_{\overline{\theta}}} \leq \widetilde{c} t^\gamma\max\{ t^{\delta-\overline{\theta}-\gamma}, t^{-1-\overline{\theta}+\theta}\}
\end{equation}

We estimate the right hand side in \eqref{acotvx01} as follows:
If $\max\{ t^{\delta-\overline{\theta}-\gamma}, t^{-1-\overline{\theta}+\theta}\}=t^{\delta-\overline{\theta}-\gamma}$, then\\
$t^\gamma\max\{ t^{\delta-\overline{\theta}-\gamma}, t^{-1-\overline{\theta}+\theta}\}=t^\gamma t^{\delta-\overline{\theta}-\gamma}=t^{\delta-\overline{\theta}},$  where $\delta-\overline{\theta}>0$. 
If, by the contrary   \linebreak
$\max\{ t^{\delta-\overline{\theta}-\gamma}, t^{-1-\overline{\theta}+\theta}\}=t^{-1-\overline{\theta}+\theta}$, then
$
t^\gamma\max\{ t^{\delta-\overline{\theta}-\gamma}, t^{-1-\overline{\theta}+\theta}\}=t^{\gamma-1-\overline{\theta}+\theta}
$ and $\gamma-1-\overline{\theta}+\theta>0$ if we take $\gamma \in \left(1-\frac{1}{2(1+\al)} +\delta,1 \right)$. Thus,
\begin{equation}\label{acotpowert3}
\norm{t^\gamma\partial^\al v_x}_{{_0}H^{(1+\al)\overline{\theta}}}=\norm{t^\gamma\partial^\al v_x}_{[L^{2}(0,1),\dd]_{\overline{\theta}}} \leq C, \quad \forall t\in [0,T].
\end{equation}
and by applying Corollary \ref{eqc} it yields that
\begin{equation}\label{acotvxforf}
\norm{t^\gamma v_x}_{{_0}H^{(1+\al)\overline{\theta}+\al}} \leq C_1 \norm{t^\gamma\partial^\al v_x}_{{_0}H^{(1+\al)\overline{\theta}}}\leq C_2, \quad \forall t\in [0,T].
\end{equation}
From \eqref{acotvxforf}, denoting $\gamma_0:=(1+\al)\overline{\theta}$ we can deduce that   $t^\gamma f \in L^\infty(0,T;{_0}H^{\gamma_0+\al} (0,1))$, for every $\gamma_0\in (0,\al)$. On the other hand, since $\al\leq\frac{1}{2}$ from (\ref{reggh}) we deduce that $g\in L^\infty(0,T;{_0}H^{\gamma_0+\al} (0,1))$, then $t^\gamma g\in L^\infty(0,T;{_0}H^{\gamma_0+\al} (0,1))$ and we conclude that
\begin{equation}
t^\gamma [f+g]\in L^\infty(0,T;{_0}H^{\gamma_0+\al} (0,1)), \quad \forall \gamma_0 \in (0,\al),\,\text{if } \gamma \in \left(1-\frac{1}{2(1+\al)}+\delta,1 \right).
\end{equation}

or equivalently,
\begin{equation}\label{regtfg}
t^\gamma[f+g]\in L^\infty(0,T;{_0}H^{\gamma_1}(0,1)),\quad \forall \gamma_1\in (\al,2\al), \,\text{ if } \gamma \in \left(1-\frac{1}{2(1+\al)}+\delta,1 \right).
\end{equation}

Now we take an arbitrary $\gamma_1<\frac{1}{2}$ (due to the fact that  $\frac{1}{4}<\al\leq \frac{1}{2}$) and  we set  $\delta_1=\frac{\gamma_1}{1+\al}<\frac{1}{2(1+\al)}$. Proceeding as in \eqref{acotAt01} and \eqref{acotAt03} we conclude that, for every $\overline{\theta}<\delta_1$,
\begin{equation}\label{acotAGfg}
\begin{split}
\norm{A(t)\izt G(t,\sigma)[f(\cdot,\sigma)+g(\cdot,\sigma)]d\sigma}_{[L^{2}(0,1),\dd]_{\overline{\theta}}} \leq\\
\leq c B( \delta_1-\overline{\theta}, 1-\gamma) t^{ \delta_1-\overline{\theta}-\gamma}\norm{t^{\gamma}(f+g)}_{L^{\infty}(0,T;[L^{2}(0,1),\dd]_{ \delta_1})},
\end{split}
\end{equation}
where we have taken a particular $\gamma$ in the last estimate such that \eqref{regtfg} holds. Besides,  we have that
\begin{equation}\label{acotAGv0}
\norm{A(t)G(t,0)v_{0}}_{[L^{2}(0,1), \dd]_{\overline{\theta}}} \leq  ct^{-1-\overline{\theta}+ \theta}\norm{v_{0}}_{[L^{2}(0,1), \dd]_{\theta}}
\end{equation}
from which, we deduce by the standard arguments, that
$$\partial^\al v_x \in L^{\infty}_{loc}(0,T;H^{(1+\al)\overline{\theta}}(0,1)), \quad \forall \overline{\theta} < \delta_1,$$
or equivalently, by denoting $\overline{\gamma}=\overline{\theta}(1+\al)$, we have that
$$\partial^\al v_x \in L^{\infty}_{loc}(0,T;H^{\overline{\gamma}}(0,1)), \quad \forall \overline{\gamma} \in (0,\gamma_1). $$
Since $\gamma_1$  is an arbitrary number in $\left(0,\frac{1}{2}\right)$, applying Corollary  \ref{eqc},
 we deduce \eqref{vxb}.

 
\item Let $\gamma$ satisfy \eqref{regtfg}. Let $\gamma_1=\frac{1}{2}-\frac{(1-\gamma)(1+\al)}{2}$, and $\delta_1=\frac{\gamma_1}{1+\al}=\frac{1}{2(1+\al)}-\frac{1-\gamma}{2}$.

Then, defining $\hat{\gamma}:=\delta_1+1-\gamma = \frac{1}{2(1+\al)} - \frac{\gamma-1}{2}>\frac{1}{2(1+\al)}$, and applying (\ref{r4}), we have that
\begin{equation}\label{acotv01}
\begin{split}
\norm{v}_{[L^{2}(0,1),\dd]_{\hat{\gamma}}} \leq \norm{G(t,0)v_0}_{[L^{2}(0,1),\dd]_{\hat{\gamma}}}+
\izt\norm{G(t,\sigma)[f(\cdot,\sigma)+g(\cdot,\sigma)]}_{[L^{2}(0,1),\dd]_{\hat{\gamma}}}d\sigma\\
\leq c_1\norm{v_0}_{[L^{2}(0,1),\dd]_{\hat{\gamma}}}+c_2 B(1-\gamma,1+\delta_1-\hat{\gamma})t^{1 -\hat{\gamma}+\delta_1-\gamma }\norm{t^\gamma[f+g]}_{L^\infty(0,T;[L^{2}(0,1),\dd]_{\delta_1})}.
\end{split}
\end{equation}

By definition of $\hat{\gamma}$, we have that $1 -\hat{\gamma}+\delta_1-\gamma = 0$. Then, \eqref{reggammahat} holds.

\item Let us denote $\delta_{1} = \frac{\gamma_{1}}{1+\al}$ where $\gamma_1 \in \left(\frac{1}{2}, 2\al\right)$ is a constant from \eqref{regtfg}, and let $\overline{\theta}<\delta_1$. Then
\begin{equation}
\begin{split}
\norm{A(t)v(\cdot,t)}_{H^{(1+\al)\overline{\theta}}(\ve,\omega)} &\leq \norm{A(t)G(t,0)v_0}_{H^{(1+\al)\overline{\theta}}(\ve,\omega)}\hspace{-0.1cm}+\hspace{-0.1cm}\int_0^t\hspace{-0.3cm} \norm{A(t)G(t,\sigma)(f+g)(\cdot,\sigma)}_{H^{(1+\al)\overline{\theta}}(\ve,\omega)} d\sigma
\end{split}
\end{equation}

and applying (\ref{r2})
\begin{equation}
\norm{A(t)G(t,0)v_0}_{H^{(1+\al)\overline{\theta}}(\ve,\omega)} \leq \norm{A(t)G(t,0)v_0}_{[L^2(0,1),\dd]_{\bar{\theta}}}\leq c t^{-1-\bar{\theta}+\theta}\norm{v_0}_{[L^2(0,1),\dd]_{\theta}}
\end{equation}

Now we consider $\gamma$ such that \eqref{regtfg} holds and we apply Lemma \ref{lemusi2} with $\beta=\gamma_1$, $\beta_1=\bar{\theta}(1+\al)$, $\bar{\gamma}=\al$ asking $\bar{\theta}$ to verify that $\bar{\theta}(1+\al)>\al$.
\begin{equation}\label{acotHloc02}
\begin{split}
\int_0^t \norm{A(t)G(t,\sigma)(f+g)(\cdot,\sigma)}_{H^{(1+\al)\overline{\theta}}(\ve,\omega)} d\sigma \hspace{7cm}\\
\leq c_2 \int_0^t  (t-\sigma)^{-1+\delta_1-\overline{\theta}}\sigma^{-\gamma}\left( \norm{\sigma^{\gamma}(f+g)(\cdot,\sigma)}_{H^{(1+\al)\delta_1}\left(\frac{\ve}{2},\frac{1+\omega}{2}\right)} + \norm{\sigma^{\gamma}(f+g)(\cdot,\sigma)}_{H^{\al}(0,1)}\right)d\sigma\\
\leq c_2 \int_0^t  (t-\sigma)^{-1+\delta_1-\overline{\theta}}\sigma^{-\gamma}d\sigma\left( \norm{t^{\gamma}(f+g)}_{L^\infty(0,T;H^{(1+\al)\delta_1}(0,1))} + \norm{t^{\gamma}(f+g)}_{L^\infty(0,T;H^{\al}(0,1))}\right).\end{split}
\end{equation}

Then, taking into account that $\overline{\theta}<\delta_1$, we have
\begin{equation}
\int_0^t \norm{A(t)G(t,\sigma)(f+g)(\cdot,\sigma)}_{H^{(1+\al)\overline{\theta}}(\ve,\omega)} d\sigma 
\leq c_3 B(\delta_1-\overline{\theta},1-\gamma)t^{\delta_1-\overline{\theta}-\gamma}
\end{equation}
and hence
\[
A(t)v \in L^{\infty}_{loc}(0,T;H^{(1+\al)\overline{\theta}}_{loc}(0,1))
 \m{ for every } \frac{\al}{1+\al}< \overline{\theta} < \delta_{1}< \frac{2\al}{1+\al}.
\]
Thus,
\[
\p^{\al}v_{x} \in L^{\infty}_{loc}(0,T;H^{(1+\al)\overline{\theta}}_{loc}(0,1)), \hd
 \m{ for every $\bar{\theta}$ such that }  \al< \overline{\theta}(1+\al) < \gamma_{1},
\]
where $\gamma_1$ was taken greater than $\frac{1}{2}$. Then there exists $\beta>\frac{1}{2}$ such that $\p^{\al}v_{x} \in L^{\infty}_{loc}(0,T;H^{\beta}_{loc}(0,1))$. 

Applying Lemma \ref{local} we get that
\[
v_{x} \in L^{\infty}_{loc}(0,T;H^{\beta+\al}_{loc}(0,1)).
\]
and
\[
v \in L^{\infty}_{loc}(0,T;H^{\beta+\al+1}_{loc}(0,1)).
\]

\end{enumerate}


{\bf Case 3: $\al\in \left(\frac{1}{2(k+1)},\frac{1}{2k}\right]$, with $k\in \bbN$, $k\geq 2$.} 

Let us start with the regularity of the term $t^\gamma f$. We first note that \eqref{regtfg} is valid for every $\al \in \left(0,\frac{1}{2}\right)$. Next we consider  $\delta_1=\frac{\gamma_1}{1+\al}$ where $\gamma_1$ is a constant in \eqref{regtfg}. Then the estimations \eqref{acotAt01} and \eqref{acotAt03} are valid for every $\bar{\theta}<\delta_1$, where in this case $\bar{\theta}<\frac{\gamma_1}{1+\al}<\frac{2\al}{1+\al}$. After that, we note that \eqref{acotpowert3} holds if $\gamma-1+\theta-\bar{\theta}>0$, but since $\theta>\frac{1}{2(1+\al)}$ this condition is verified for every $\bar{\theta}< \delta_{1}$ if $\gamma \in \left(1-\frac{1}{2(1+\al)}+\delta_{1},1\right).
$

Then, using \eqref{reggh}, we can conclude that 
\begin{equation}
t^\gamma [f+g]\in L^\infty(0,T;{_0}H^{(1+\al)\bar{\theta}+\al} (0,1)), \quad \forall \bar{\theta}<\delta_1, \, \text{if } \gamma \in \left(1-\frac{1}{2(1+\al)}+\delta_{1},1 \right).
\end{equation}

Which leads to
\begin{equation}\label{regtfg-2}
t^\gamma[f+g]\in L^\infty(0,T;{_0}H^{\gamma_2}(0,1)),\quad \forall \gamma_2\in (0,3\al), \,\text{if  } \gamma \in \left(1-\frac{1}{2(1+\al)}+\delta_{1},1 \right).
\end{equation}

Repeating the previous reasoning $k$ times, we conclude that, for $\al \in \left(\frac{1}{2(k+1)},\frac{1}{2k}\right]$ with $\delta_{k} :=\frac{\gamma_{k}}{1+\al}$ we have 
\begin{equation}\label{regtfg-k}
t^\gamma[f+g]\in L^\infty(0,T;{_0}H^{\gamma_k}(0,1)),\quad \forall \gamma_k\in (0,(k+1)\al), \,\text{if  } \gamma \in \left(1-\frac{1}{2(1+\al)}+\delta_{k-1},1 \right).
\end{equation}

\begin{enumerate}
\item  Consider $\gamma_k < \frac{1}{2}$ arbitrary (which is possible due to the fact that  $\frac{1}{2(k+1)}<\al< \frac{1}{2k}$) and denote  $\delta_k=\frac{\gamma_k}{1+\al}$. Then proceeding as in item 1 of the former case we can say that  \eqref{acotAGfg} with $\delta_{k}$ instead of $\delta_{1}$, holds for every $\bar{\theta}<\delta_k$, if we take an auxiliary $\gamma$  verifying \eqref{regtfg-k}. Also we have \eqref{acotAGv0} and  then we can conclude  that
$$\partial^\al v_x \in L^{\infty}_{loc}(0,T;H^{(1+\al)\overline{\theta}}(0,1)), \quad \forall \overline{\theta} < \delta_k,$$
or equivalently, by denoting $\overline{\gamma}=\overline{\theta}(1+\al)$, we have that
$$\partial^\al v_x \in L^{\infty}_{loc}(0,T;H^{\overline{\gamma}}(0,1)), \quad \forall \overline{\gamma} \in (0,\gamma_k), \quad \forall \gamma_k<\frac{1}{2},$$
that is,
$$\partial^\al v_x \in L^{\infty}_{loc}(0,T;H^{\overline{\gamma}}(0,1)), \quad \forall \overline{\gamma} \in \left(0,\frac{1}{2}\right).$$

Then, by Corollary \ref{eqc} and \eqref{p60} we get that, for $\al\in\left(\frac{1}{2(k+1)},\frac{1}{2k}\right]$,
$$v_x \in L^{\infty}_{loc}(0,T;H^{\overline{\gamma}+\al}(0,1)), \quad \forall \overline{\gamma} \in \left(0,\frac{1}{2}\right).$$

\item It is analogous to the case $\al \in \left(\frac{1}{4},\frac{1}{2}\right]$ by considering $\gamma_k$ and $\gamma$ verifying \eqref{regtfg-k} and then choosing $\hat{\gamma}:=1-\gamma+\delta_k$.
\item It is analogous to the case $\al \in \left(\frac{1}{4},\frac{1}{2}\right]$ considering $\gamma_k$ and  $\gamma$ verifying \eqref{regtfg-k}, applying Lemma~\ref{lemusi2} and obtaining \eqref{acotHloc02}, which holds for every $\frac{\al}{\al+1}<\bar{\theta}<\delta_k$.
\end{enumerate}

\no Finally, we prove item 4, for every $\al \in (0,1)$.\\

 $4.$ Theorem \ref{first} states that $v \in C((0,T];\dd)$. Since for arbitrary $0<\ve<\omega<1$ and for every $0<\overline{\beta} < \beta$ there holds
\[
H^{\overline{\beta}+\al+1}(\ve,\omega) = [H^{1+\al}(\ve,\omega),H^{\beta+\al+1}(\ve,\omega)]_{\frac{\overline{\beta}}{\beta}},
\]
we may estimate by the interpolation theorem (\cite[Corollary 1.2.7]{Lunardi})
\[
\norm{v(\cdot,t)-v(\cdot,\tau)}_{H^{\overline{\beta}+\al+1}(\ve,\omega)}
 \leq c\norm{v(\cdot,t)-v(\cdot,\tau)}_{H^{1+\al}(0,1)}^{1-\frac{\overline{\beta}}{\beta}}
\norm{v(\cdot,t)-v(\cdot,\tau)}_{H^{\beta+\al+1}(\ve,\omega)}^{\frac{\overline{\beta}}{\beta}},
\]
where $c =  c(\beta,\overline{\beta},\ve)$.
By Lemma \ref{third}-2, the second norm on the right hand side above is bounded on every compact interval contained in $(0,T]$, while the first tends to zero as $\tau\rightarrow t$ for $t,\tau \in (0,T]$.
Then, we may choose $\bar{\beta} > \frac{1}{2}$ such that $v \in C((0,T];H^{\bar{\beta}+1+\al}_{loc}(0,1))$.

\no In order to obtain the claim for $\p^{\al}v_{x}$ we recall that by Theorem \ref{first} we have $\p^{\al}v_{x} \in C((0,T];L^{2}(0,1))$.
Applying again the interpolation theorem we obtain for every $0<\ve<\omega<1$, $0<\tau < t \leq T$ and every $0<\overline{\beta} < \beta$
\[
\norm{\p^{\al}v_{x}(\cdot,t) - \p^{\al}v_{x}(\cdot,\tau)}_{H^{\overline{\beta}}(\ve,\omega)} \hspace{-0.1cm}
\leq c(\beta,\overline{\beta},\al)\norm{\p^{\al}v_{x}(\cdot,t) - \p^{\al}v_{x}(\cdot,\tau)}_{L^{2}(0,1)}^{1-\frac{\overline{\beta}}{\beta}}\norm{\p^{\al}v_{x}(\cdot,t) - \p^{\al}v_{x}(\cdot,\tau)}_{H^{\beta}(\ve,\omega)}^{\frac{\overline{\beta}}{\beta}}.
\]
The first norm tends to zero as $\tau \rightarrow t$, while the second one is bounded on every compact interval contained in $(0,T]$ due to Lemma \ref{third}-3.

Then, we may choose $\bar{\beta} > \frac{1}{2}$ such that $\p^{\al}v_{x} \in C((0,T];H^{\bar{\beta}}_{loc}(0,1))$.

\end{proof}

Finally, we can establish the continuity of the mild solution.

\begin{coro}\label{regost}
 Let us assume that $h \in W^{1,\infty}(0,T)$, $v_{0} \in {}^{0}H^{\theta(\al+1)}(0,1)$, for fixed $\theta \in (\frac{1}{2(1+\al)},1]$.
 Then,  the solution to (\ref{niecv}) obtained in Theorem~\ref{first} satisfies that
  \eqq{
  v \in C([0,T]\times [0,1]).
  }{vciaglosc}
	 \end{coro}
\begin{proof}

The continuity of $v$ follows from the fact that by Lemma \ref{third} there holds \linebreak $v \in L^{\infty}((0,T);[\ld,\dd]_{\bar{\gamma}})$ for $\bar{\gamma}>\frac{1}{2(1+\al)}$. Since $v \in C([0,T];\ld)$ by the  interpolation argument we obtain that $v \in C([0,T];[\ld,\dd]_{\bar{\gamma}_{1}})$ for a $\frac{1}{2(1+\al)}< \bar{\gamma}_{1} < \bar{\gamma}$ and hence $v \in C([0,T]\times [0,1])$ by the Sobolev embedding.
\end{proof}

Finally, in the next Corollary  we find the regularity needed to work with extremum principles.

\begin{coro}\label{w2b}
Under the assumptions of Lemma \ref{third} for every $\al \in (0,1)$ there exists $\beta > \frac{1}{2}$ such that for every $0<\ve<\omega<1$ there holds $v_{x}~\in C((0,T];H^{\al+\beta}(\ve,\omega))$.
\end{coro}


\subsection{The existence and regularity of solutions to the MBP  (\ref{niech})}
\no Finally, we are ready to formulate and prove the result concerning the existence, uniqueness and regularity of the solution to (\ref{niech}).
\begin{theo}\label{existance}
Let $b,T > 0$ and $\al \in (0,1)$. Let us assume that $s$ satisfies (\ref{zals}). We further assume that $u_{0}(b) = 0$, $u_{0} \in H^{\theta(1+\al)}(0,b)$ for fixed $\theta \in (\frac{1}{2(\al+1)},1)$. Then, there exists a unique solution $u$ to~(\ref{niech}) such that $u \in C(\overline{Q_{s,T}})$, $u_{t},\poch \da u \in C(Q_{s,T})$ and for every $t \in (0,T]$ $u(\cdot,t) \in AC[0,s(t)], \poch \da u(\cdot,t), u_{t}(\cdot,t) \in L^{2}(0,s(t))$ . Furthermore, there exists $\beta >\frac{1}{2}$ such that for every $t \in (0,T]$ and every $0<\ve<\omega<s(t)$ we have $u_{x}(\cdot,t) \in H^{\al+\beta}(\ve,\omega)$.
\end{theo}
\begin{proof}
Firstly we will establish the results concerning the existence and regularity of solution to (\ref{niecv}) and then, we will rewrite the results in terms of properties of solution to (\ref{niech}).
We note that, under assumptions concerning regularity and traces of $\uz$ we obtain that $v_{0}$  defined in (\ref{defv0}) satisfies the assumption of Theorem~\ref{first}, Lemma~\ref{third} and Corollary~\ref{regost}.
Hence, there exists $v$ a unique solution to (\ref{niecv}) with the regularity given by Theorem~\ref{first}, Lemma~\ref{second}, Lemma~\ref{third} and Corollary~\ref{regost}. In particular,  $v \in C([0,T]\times [0,1])$. Furthermore, from Corollary~\ref{w2b} we know that there exists $\beta > \frac{1}{2}$ such that $v_{x} \in C((0,T];H^{\al+\beta}(\ve,\omega))$ for every $0<\ve<\omega<1$.

\no According to \eqref{u=reg+sing} we define the function $u$ on $Q_{s,T}$ by the formula $u(x,t) = v(\frac{x}{s(t)},t) + h(t)s^{\al}(t)\varphi(\frac{x}{s(t)})$. Since $v \in C([0,T]\times [0,1])$, we obtain that $u \in C(\overline{Q_{s,T}})$ and $v_{x} \in C((0,T];H^{\al+\beta}(\ve,\omega))$  implies $u_{x}(\cdot,t) \in H^{\al+\beta}(\ve,\omega)$ for every $t \in (0,T]$ and every $0<\ve<\omega<s(t)$. Recalling that $v \in C((0,T];\dd)$ and $\varphi \in AC[0,1]$ we obtain that $u(\cdot,t) \in AC[0,s(t)]$ for every $t \in (0,T]$.
Furthermore,  since $v$ is a unique solution to (\ref{niecv}) we obtain that $u$ satisfies (\ref{niech}). Indeed, a simple calculation shows that $u$ satisfies $(\ref{niech})_{1}$ and
\[
u_{t}(x,t) = \poch \da u(x,t) = \frac{1}{s^{1+\al}(t)} \frac{\p}{\p p} \da v(p,t) + \frac{h(t)}{s(t)}\frac{\p}{\p p} \da \varphi(p)  \m{ where } p=\frac{x}{s(t)}.
\]
From $v \in C((0,T];\dd)$ and (\ref{auxvar}) we obtain $\poch \da u(\cdot,t), u_{t}(\cdot,t) \in L^{2}(0,s(t))$ for every $t > 0$.
Applying Corollary \ref{regost} and the Sobolev embedding we may deduce that $\p^{\al}v_{p} =\frac{\p}{\p p} \da v  \in C((0,T] \times (0,1))$, which implies $\poch \da u,\hd u_{t} \in C(Q_{s,T})$.
Finally,
\[
\da u(x,t) = \frac{1}{s^{\al}(t)}\da v(p,t) + h(t)\da \varphi(p).
\]
By definition of $\vf$ we have $\da \varphi(0) = 1$ and since $v \in C((0,T];\dd)$ we arrive at
\[
\da v(p,t) = I^{1-\al}v_{p}(p,t) \in {}_{0}H^{1}(0,1) \m{ for every } t > 0,
\]
which implies $\da u(0,t) = h(t)$ for every $t > 0$. A uniqueness follows from a standard energy estimate.
\end{proof}

\subsection{A solution to the Stefan problem}
In this section, we prove existence of solutions to the associated free-boundary problem for the $b>0$ case, which is to find the pair of functions $(u,s)$ satisfying the system

\begin{equation}\label{Stefan}
 \left\{ \begin{array}{ll}
u_{t} - \poch \da u = 0 & \textrm{ in } \{(x,t):0<x<s(t), 0<t<T\}=: Q_{s,T}, \\
\da u(0,t) =h(t), \ \ u(t,s(t)) = 0 & \textrm{ for  } t \in (0,T), \\
u(x,0) = u_{0}(x) & \textrm{ for } 0<x<s(0)=b, \\
\dot{s}(t) = -(\da u)(s(t),t) & \textrm{ for  } t \in (0,T).
\end{array} \right. \end{equation}

\begin{theo}\label{finala}
Let $b, T > 0$ and $\al \in (0,1)$. Let us assume that, $h$ satisfies (\ref{zalh}),  $\uz$ satisfies the assumptions of Theorem \ref{existance} and Proposition \ref{daboundlem}. 
Then, there exists exactly one  solution $(u,s)$  to (\ref{Stefan}) in the sense of Definition \ref{sol-b>0}, satisfying: there exists $\beta > \frac{1}{2}$, such that for every $t \in (0,T]$ and every $0<\ve<\omega<s(t)$ there holds $u_{x}(\cdot,t)\in H^{\al+\beta}(\ve,\omega)$. Finally, if $\uz \not\equiv 0$ or $h \not\equiv 0$ then $0 < \dot{s}$ a.e. on $[0,T]$.
\end{theo}
\begin{proof}

We follow the idea introduced in the proof of \cite[Theorem 5.1]{Andreucci}.
We define the set
\eqq{
\Sigma := \{s\in C^{0,1}[0,T], \hd 0 \leq \dot{s}\leq M, \hd s(0)=b \}
}{Sigma}
where $M$ is determined by the Proposition \ref{daboundlem}, i.e., $M=||h||_\infty$.
Then for every $s\in \Sigma$ there exists a unique solution to (\ref{niech}), given by Theorem \ref{existance}. We note that $\Sigma$ is a compact and convex subset of a Banach space $C([0,T])$ with a maximum norm.
For $s\in \Sigma$ we define the operator
\[
(Ps)(t) = b - \izt (\da u)(s(\tau),\tau)\rmd \tau,
\]
where $u$ is a solution to (\ref{niech}), corresponding to $s$, given by Theorem \ref{existance}. We would like to apply the Schauder fixed point theorem, thus we have to show that $P:\Sigma\rightarrow\Sigma$ and that it is continuous in maximum norm. Clearly we have $(Ps)(0) = b$ and from Proposition \ref{nonposda} and estimate (\ref{dabound}) we infer \linebreak $
0 \leq \frac{d}{dt}(Ps)(t) = -(\da u)(s(t),t) \leq M.
$ Hence, $P:\Sigma\rightarrow\Sigma$.

\no To prove that $P$ is continuous in maximum norm, at first we make use of \eqref{Equiv-1} to arrive at
\eqq{
(Ps)(t) = \izt h(\tau) \rmd \tau + b + \int_{0}^{b}\uz(x)\rmd x - \int_{0}^{s(t)}u(x,t) \rmd x.
}{pfo}

Now, we take arbitrary $s_{1}$, $s_{2} \in \Sigma$. Let us define $s_{min}(t)=\min\{s_{1}(t), s_{2}(t)\}$,\linebreak $s_{max}(t)=\max\{s_{1}(t), s_{2}(t)\}$. We also define function $i=i(t) = 1$ if $s_{max}(t) = s_{1}(t)$ and $i = 2$ otherwise. Let $u_{1}$ and $u_{2}$ be two solutions to (\ref{niech}), given by Theorem \ref{existance}, corresponding to $s_{1}$ and $s_{2}$ respectively.
Let us define $z(x,t) = u_{1}(x,t)-u_{2}(x,t)$ and $z^{\ve}(x,t) = z(x,t)+\ve \frac{x^{\al}}{\Gamma(1+\al)}$. Then $z^{\ve}$ satisfies
 \[
 \left\{ \begin{array}{ll}
z^{\ve}_{t} - \poch \da z^{\ve} = 0 & \textrm{ in } \{(x,t):0<x<s_{min}(t), 0<t<T\}=: Q_{s_{min},T}, \\
\da z^{\ve}(0,t) = \ve,  & \textrm{ for  } t \in (0,T), \\
z^{\ve}(x,0) = \ve \frac{x^{\al}}{\Gamma(1+\al)} & \textrm{ in } 0<x<b. \\
\end{array} \right. \]

\no From Theorem \ref{weakPrinc} we obtain that $z^{\ve}$ attains its maximum on the parabolic boundary. We may estimate
\[
\abs{z^{\ve}(s_{min}(t),t)} \leq \abs{u_{1}(s_{min}(t),t)} + \abs{u_{2}(s_{min}(t),t)} + \ve \frac{s_{min}^{\al}(T)}{\Gamma(1+\al)} = \abs{u_{i}(s_{min}(t),t)}+ \ve \frac{s_{min}^{\al}(T)}{\Gamma(1+\al)}
\]
and since $z^{\ve}(x,0) = \ve \frac{x^{\al}}{\Gamma(1+\al)} \leq \ve\frac{s_{min}^{\al}(T)}{\Gamma(1+\al)}$ and $\da z^{\ve}(0,t) > 0$ we obtain that
\[
\max_{Q_{s_{min},T}}z^{\ve} \leq \abs{u_{i}(s_{min}(t),t)}+ \ve \frac{s_{min}^{\al}(T)}{\Gamma(1+\al)}.
\]
Then, applying the estimate (\ref{estu}) from Proposition \ref{daboundlem} we get
\begin{equation}
\begin{split}
\abs{u_{i}(s_{min}(t),t)} &\leq \frac{M}{\Gamma(1+\al)}\left(s_{max}^\al(t)-s_{min}^\al(t)\right) \leq \frac{M}{\Gamma(1+\al)}\frac{s_{max}^\al(t)-s_{min}^\al(t)}{s_{max}(t)-s_{min}(t)}\left(s_{max}(t)-s_{min}(t)\right) \\
 &\leq \frac{M}{\Gamma(1+\al)} s_{max}^{\al-1}(t)\frac{1-\left(\frac{s_{min}(t)}{s_{max}(t)}\right)^\al}{1-\frac{s_{min}(t)}{s_{max}(t)}}  \max_{\tau\in [0,t]}\abs{s_{1}(\tau)-s_{2}(\tau)}\leq  \frac{M b^{\al-1} }{\Gamma(1+\al)} \max_{\tau\in [0,t]}\abs{s_{1}(\tau)-s_{2}(\tau)}
\end{split}
\end{equation}
Hence,
\[
\max_{Q_{s_{min},T}} z = \max_{Q_{s_{min},T}}(z^{\ve}-\ve \frac{x^{\al}}{\Gamma(1-\al)}) \leq \frac{M b^{\al-1} }{\Gamma(1+\al)}\max_{\tau\in [0,t]}\abs{s_{1}(\tau)-s_{2}(\tau)} + \ve \frac{s_{min}^{\al}(T)}{\Gamma(1+\al)}.
\]
Passing with $\ve$ to zero we obtain
\[
\max_{Q_{s_{min},T}}z \leq \frac{M b^{\al-1} }{\Gamma(1+\al)}\max_{\tau\in [0,t]}\abs{s_{1}(\tau)-s_{2}(\tau)}.
\]

To estimate $z$ from below we proceed similarly by  introducing $z_{\ve}(x,t) = z(x,t) - \ve \frac{x^{\al}}{\Gamma(1+\al)}$. Together we obtain 
\[
\max_{Q_{s_{min},T}}\abs{z}\leq \frac{M b^{\al-1} }{\Gamma(1+\al)}\max_{\tau\in [0,t]}\abs{s_{1}(\tau)-s_{2}(\tau)}.
\]
Furthermore, estimate (\ref{estu}) implies that
\[
\int_{s_{min}(t)}^{s_{max}(t)}u_{i}(x,t)\rmd x \leq \frac{M b^{\al-1} }{\Gamma(1+\al)} \int_{s_{min}(t)}^{s_{max}(t)}(s^\al_{max}(t) - x^\al)\rmd x
\leq  \frac{M b^{\al-1} }{\Gamma(1+\al)}s^\al_{max}(T)(s_{max}(t) - s_{min}(t)).
\]
Finally, we may estimate
\[
\abs{(Ps_{2})(t)-(Ps_{1})(t)} = \abs{\int_{0}^{s_{2}(t)}u_{2}(x,t)\rmd x - \int_{0}^{s_{1}(t)}u_{1}(x,t)\rmd x}
\]
\[
\leq \int_{0}^{s_{min}(t)}\abs{ z(x,t)} \rmd x + \int_{s_{min}(t)}^{s_{max}(t)}u_{i}(x,t)\rmd x
\]
\[
\leq s_{min}(t)\max_{Q_{s_{min},T}}\abs{z}+ \frac{M b^{\al-1} }{\Gamma(1+\al)}s^\al_{max}(T)\max_{\tau\in[0,t]}\abs{s_{1}(\tau) - s_{2}(\tau)}
\]
\[
\leq (b+MT)\frac{M b^{\al-1} }{\Gamma(1+\al)}\max_{\tau\in[0,t]}\abs{s_{1}(\tau) - s_{2}(\tau)} +\frac{M b^{\al-1} }{\Gamma(1+\al)}s^\al_{max}(T)\max_{\tau\in[0,t]}\abs{s_{1}(\tau) - s_{2}(\tau)}\]
\[ \leq \tilde{C}\max_{\tau\in[0,t]}\abs{s_{1}(\tau) - s_{2}(\tau)}.
\]
Thus $P$ is continuous and by the Schauder fixed point theorem there exists a fixed point of~$P$.
In this way we have proven that there exists a pair $(u,s)$ that satisfies the system (\ref{Stefan}), where $s \in \Sigma$ and $u$ is given by Theorem \ref{existance}. A uniqueness of the solution follows from Theorem \ref{uniN}. We finish the proof applying Proposition~\ref{nonposda}.
\end{proof}

\section{The solution for the $b=0$ case.}

This chapter is devoted to  problem  (\ref{IFSSP}), that is, the $b=0$ case and we recall it for the benefit of the reader:  Find the pair of functions $u\colon Q_{s,T}\rightarrow \bbR$ and $s\colon [0,T]\rightarrow \bbR$ sufficiently regular, such that
\begin{equation}\label{IFSSP-bis}
\begin{array}{lll}
(i)& u_t(x,t)=\frac{\p}{\p x} \, D_x^{\al} u(x,t) & 0<x<s(t), 0<t<T,\\
(ii) & -D_x^{\al} u(0^+,t)=h(t)\geq 0  &  0<t<T,\\
(iii) & s(0)=0 & \\
(iv) & u(s(t),t)=0 & 0<t<T,\\
(v) & \dot{s}(t)=-\lim\limits_{x\rightarrow s(t)^-} \,D_x^{\al} u(x,t) & 0<t<T.
\end{array}
\end{equation}

We find a solution to problem \eqref{IFSSP-bis} by an approximation method. Precisely, we approximate \eqref{IFSSP-bis} by a sequence  of problems (\ref{FSSP-al-N-2}) with $b = 1/m$ for $m \in \mathbb{N}$. Hence, the $b$ - independent energy estimates for the MBP problem (\ref{niech}) will be the essential tool, which is developed in the next subsection.

\subsection{Energy estimates}
We firstly  recall a useful result from \cite{KY}.

\begin{prop}\cite[Proposition 6.10]{KY}
If $w \in AC[0, L]$, then for any $\al \in (0,1)$ the following equality holds
\[
\int_{0}^{L} \p^{\al}w(x) \cdot w(x)dx = \frac{\al}{4} \int_{0}^{L} \int_{0}^{L} \frac{\abs{w(x)-w(p)}^{2}}{\abs{x-p}^{1+\al}}\rmd p \rmd x
\]
\[
+\frac{1}{2\Gamma(1-\al)}\int_{0}^{L} [(L-x)^{-\al} + x^{-\al}]\abs{w(x)}^{2} \rmd x.
\]
Hence, there exists a positive constant $c(\al)$ which depends only on $\al$, such that
\eqq{
\int_{0}^{L} \p^{\al}w(x) \cdot w(x)\rmd x  \geq c(\al)[1+L^{-\al}] \norm{w}_{H^{\frac{\al}{2}}(0,L)}^{2}.
}{AH}
\end{prop}
\begin{remark}\label{reAH}
  We note that the proposition above holds also for $w \in H^{\frac{\al}{2}}(0,L)$ if we understand the integral $\int_{0}^{L} \p^{\al}w(x) \cdot w(x)dx $ as a duality pairing in $H^{-\frac{\al}{2}}(0,L) \times H^{\frac{\al}{2}}(0,L)$. Indeed, in view of \cite[Proposition 5]{ja} we have that $\p^{\al}w \in H^{-\frac{\al}{2}}(0,L)$ and the result follows by density argument.
\end{remark}
We will begin with the first energy estimate, thus we will multiply the equation  $(\ref{IFSSP-bis}-i)$ by a solution. We will formulate the result in terms of general solutions to a moving boundary problem (\ref{niech}). Since in a subsequent section, we will be interested in uniform estimates with respect to small $b > 0$, here we assume that $b \in (0,1]$. 
\begin{lemma}\label{estfirst-Lemma}{(A first energy estimate)}
Let  $h$ and $s$ verify \eqref{zalh} and \eqref{zals} respectively  and let $u$ be the solution to problem \eqref{niech} with $b \in (0,1]$ given by Theorem \ref{existance}. Then, there exists a positive constant $c = c(\al,M,T)$, where $M$ comes from (\ref{zals}) such that for every $t > 0$
\eqq{
\int_{0}^{s(t)} \abs{u(x,\tau)}^{2}\rmd x +  \izt \norm{u(\cdot,\tau)}^{2}_{H^{\frac{1+\al}{2}}(0,s(\tau))}\rmd \tau \leq  c(\al,M,T)(\norm{h}_{L^{2}(0,t)}^{2} + \norm{u_{0}}^{2}_{L^{2}(0,b)}).
}{estfirst}
\end{lemma}

\begin{proof}
In order to get (\ref{estfirst}) we multiply $(\ref{niech}-i)$ by $u$ and integrate from zero to $s(t)$
\[
\int_{0}^{s(t)}u_{t}(x,t)\cdot u(x,t)\rmd x - \int_{0}^{s(t)}\poch\da u(x,t)\cdot u(x,t)\rmd x = 0.
\]
Integrating by parts we arrive at
\[
\frac{1}{2}\int_{0}^{s(t)} \frac{d}{dt}\abs{u(x,t)}^{2}\rmd x +h(t)u(0,t)+ \int_{0}^{s(t)}\da u(x,t)\cdot u_{x}(x,t)\rmd x = 0.
\]
We note that by (\ref{zals}) there holds $(s(t))^{-\al} \geq (b+MT)^{-\al}$ and applying $b \in (0,1]$ we obtain that
\eqq{
(s(t))^{-\al} \geq (1+MT)^{-\al} \m{ for every } t \in (0,T].
}{sbind}
We note that by Theorem~\ref{existance}, for every $t \in (0,T)$ $\poch \da u(\cdot,t), u_{t}(\cdot,t) \in L^{2}(0,s(t))$. Hence, in particular, for every $t \in (0,T)$ $\da u (\cdot,t) \in AC[0,s(t)]$, applying estimate (\ref{AH}) together with (\ref{sbind}) we obtain that there exists $c = c(\al,M,T)$ such that 
\[
\int_{0}^{s(t)}\da u(x,t)\cdot u_{x}(x,t)\rmd x = \int_{0}^{s(t)}\da u(x,t)\cdot \p^{1-\al}\da u(x,t)\rmd x \geq c \norm{\da u(\cdot,t)}^{2}_{H^{\frac{1-\al}{2}}(0,s(t))}
\]
\[
\geq c \norm{\p^{\frac{1-\al}{2}}\da u(\cdot,t)}^{2}_{L^{2}(0,s(t))}
= c \norm{D^{\frac{\al+1}{2}}u(\cdot,t)}^{2}_{L^{2}(0,s(t))},
\]
where in the second estimate we applied Proposition \ref{eq}.
Despite that $u_{x} \not \in {}_{0}H^{\al}(0,s(t))$, we may repeat the proof of an inequality
\eqq{
\norm{D^{\frac{\al+1}{2}}u(\cdot,t)}_{L^{2}(0,s(t))} \geq c \norm{u(\cdot,t)}_{H^{\frac{1+\al}{2}}(0,s(t))}
}{rowno}
from \cite[Lemma 2]{ja}. Indeed, since $u(s(t),t) = 0$ we have
\[
u(x,t)=-\int_{x}^{s(t)} u_{x}(p,t)\rmd p = -(I_{-}u_{x})(x,t) =- (I_{-}^{\frac{1+\al}{2}}I_{-}^{\frac{1-\al}{2}}u_{x})(x,t),
\]
where the right-side operators where defined in Remark \ref{reright}. From Proposition \ref{eq} we may estimate
\[
\norm{ u}_{H^{\frac{\al+1}{2}}(0,s(t))} = \norm{I^{\frac{1+\al}{2}}_{-}I^{\frac{1-\al}{2}}_{-}u_{x}}_{{}^{0}H^{\frac{\al+1}{2}}(0,s(t))} \leq c\norm{I^{\frac{1-\al}{2}}_{-}u_{x}}_{L^{2}(0,s(t))},
\]
where 
$
{}^{0}H^{\frac{\al+1}{2}}(0,s(t)) =[L^{2}(0,s(t)),{}^{0}H^{1}(0,s(t))], \hd {}^{0}H^{1}(0,s(t)):=\{ f\in H^{1}(0,s(t)), f(s(t)) = 0  \}.
$

\no Applying Proposition 4 and Proposition 5 from \cite{ja} we may estimate further
\[
\norm{ u}_{H^{\frac{\al+1}{2}}(0,s(t))} \leq c \norm{ u_{x}}_{H^{\frac{\al-1}{2}}(0,s(t))} = c\norm{\p^{\frac{1-\al}{2}}I^{\frac{1-\al}{2}} u_{x}}_{H^{\frac{\al-1}{2}}(0,s(t))} \leq c \norm{I^{\frac{1-\al}{2}} u_{x}}_{L^{2}(0,s(t))}
\]
and we arrive at (\ref{rowno}).
Thus, together we obtain that
\[
\frac{1}{2}\int_{0}^{s(t)} \frac{d}{dt}\abs{u(x,t)}^{2}\rmd x + c \norm{u(\cdot,t)}^{2}_{H^{\frac{1+\al}{2}}(0,s(t))} \leq -h(t)u(0,t),
\]
where $c =c(\al,M,T)$.
We integrate from zero to $t$ to obtain that
\[
\frac{1}{2} \izt \int_{0}^{s(\tau)} \frac{d}{d\tau}\abs{u(x,\tau)}^{2}\rmd x \rmd\tau +c \izt \norm{u(\cdot,\tau)}^{2}_{H^{\frac{1+\al}{2}}(0,s(\tau))}\rmd\tau \leq \abs{\izt h(\tau)u(0,\tau)d\tau}.
\]
Since $u(s(t),t) = 0$ we arrive at
\[
\frac{1}{2}\int_{0}^{s(t)} \abs{u(x,t)}^{2}\rmd x + c \izt \norm{u(\cdot,\tau)}^{2}_{H^{\frac{1+\al}{2}}(0,s(\tau))}\rmd \tau \leq \ve \norm{u(0,\cdot)}_{L^{2}(0,t)}^{2} + \frac{1}{4 \ve}\norm{h}_{L^{2}(0,t)}^{2}+\frac{1}{2}\norm{\uz}_{L^{2}(0,b)}^{2}.
\]
Applying the Sobolev embedding  and taking $\ve = \ve(c)$ small enough we arrive at
(\ref{estfirst}).
\end{proof}

Now we look for higher order estimates on $u$. However, since $u$ has a singularity $x^{\al}$ we will obtain the second energy estimate for the regular part of $u$ defined in \eqref{u=reg+sing}. Notice that, in the aim to simplify the notation further we will denote $u_{reg}$ by $v$, that is to say,
\eqq{
u = v + u_{sin},
}{urep}
where
\[
u_{sin}=
 \left\{ \begin{array}{ll}
h(t)\frac{x^{\al}}{\Gamma(1+\al)} & \textrm{ in } \{(x,t):0<x\leq \frac{1}{2}s(t), 0<t<T\}, \\
\bar{\varphi} & \textrm{ in } \{(x,t):\frac{1}{2}s(t)<x\leq\frac{3}{4}s(t), 0<t<T\}, \\
0 &  \textrm{ in } \{(x,t):\frac{3}{4}s(t)<x\leq s(t), 0<t<T\},\\
\end{array} \right.
\]
where $\bar{\varphi}(x,t)=h(t)s^{\al}(t)\varphi\left(\frac{x}{s(t)}\right)$ for the corresponding auxiliary  function $\varphi$. Recall that, from  \eqref{auxvar}, it holds  that $\bar{\varphi}(\frac{s(t)}{2},t) = h(t)\frac{(s(t)/2)^{\al}}{\Gamma(1+\al)}$ and $\bar{\varphi}(\frac{3s(t)}{4},t)~=~0$.
Besides,  $v$ has the regularity established in Lemma~\ref{third} and Corollary~\ref{regost},  and  
from (\ref{vxb}) it holds that  $v^{m}_{xx} \in H^{\al+\bar{\gamma}-1}(0,s(t))$ for a.a $t \in (0,T]$, for every $0<\bar{\gamma}<\frac{1}{2}$.
Here for simplicity we discuss the case of zero initial condition and, as in the previous lemma, we consider $b \in (0,1]$.
\begin{lemma}\label{szac2-lemma}{(A second energy estimate for $u_{reg}$)}
If $u$ is a solution to problem (\ref{niech}) with zero initial condition and $b \in (0,1]$, given by Theorem~\ref{existance}, then $v$ - the regular part of $u$ given by (\ref{urep}) satisfies the following bound for every $t_{1} \in (0,T)$ and $t>t_1$,
\eqq{
\int_{t_{1}}^{t}\dot{s}(\tau)\abs{v_{x}(s(\tau),\tau)}^{2}\rmd\tau
+\int_{0}^{s(t)}\hspace{-0.2cm} \abs{v_{x}(x,t)}^{2}\rmd x
+  \int_{t_{1}}^{t} \norm{ v_{x}(\cdot,\tau)}^{2}_{H^{\frac{1+\al}{2}}(0,s(\tau))}\hspace{-0.1cm} \rmd\tau
 \leq c(\al,t_{1}, T, M,\norm{h}_{W^{1,\infty}(0,T)}).
}{szac2}
\end{lemma}
\begin{proof}
In view of (\ref{urep}), we know that $v$ satisfies in $Q_{s,T}$ the following identity
\eqq{
v_{t} - \poch \da v = -u_{sin, t} + \poch \da u_{sin}.
}{vfmeq}
Let us denote
\[
\bar{g}:= -u_{sin, t} + \poch \da u_{sin}.
\]
We multiply (\ref{vfmeq}) by $v_{xx}$ which belongs to $H^{\al+\bar{\gamma}-1}(0,s(t))$ for every $0<\bar{\gamma} < \frac{1}{2}$. If its necessary we understand the integrals as a duality pairing
\eqq{
\int_{0}^{s(t)} v_{t} \cdot v_{xx}\rmd x - \int_{0}^{s(t)} \poch \da v\cdot v_{xx} \rmd x = \int_{0}^{s(t)} \bar{g} \cdot v_{xx}\rmd x.
}{secest}
We note that since $ u_{sin}(\frac{1}{2}s(t),t) = h(t)\frac{(s(t))^{\al}}{2^{\al}\Gamma(1+\al)}$ and $ u_{sin}\equiv 0$ for $x \geq \frac{3}{4}s(t)$, we may choose the function $\vf$ in (\ref{auxvar}) such that $\abs{ u_{sin, x}(x,t)} \leq c h(t)(s(t))^{(\al-1)}$ and analogously $\abs{D_{x}^{(k)} u_{sin}(x,t)} \leq c h(t)(s(t))^{\al-k}$ for $k \in \mathbb{N}$. However we are unable to estimate the negative powers of $s(t)$ by $b$ - independent constant. That is why we multiply identity (\ref{secest}) by a smooth cut-off function $\Psi:[0,T]\rightarrow \mathbb{R}$ such that $\Psi\equiv 0$ on $[0,t_{1}]$, $\Psi\equiv 1$ on $[2t_{1},T]$, where we fixed arbitrary $0<t_{1}<T$. Then we have
\eqq{
-\izt\Psi(\tau)\int_{0}^{s(\tau)} v_{t} \cdot v_{xx}\rmd x \rmd\tau + \izt \Psi(\tau)\int_{0}^{s(\tau)} \poch \da v \cdot v_{xx} \rmd x \rmd\tau =- \izt \Psi(\tau) \int_{0}^{s(\tau)} \bar{g}\cdot v_{xx}\rmd x \rmd\tau.
}{secest2}

Let us proceed term by term
\[
-\izt \Psi(\tau) \int_{0}^{s(\tau)} v_{t}(x,\tau) \cdot v_{xx}(x,\tau)\rmd x \rmd\tau
\]
\[
=-\izt \Psi(\tau)v_{t}(s(\tau),\tau)v_{x}(s(\tau),\tau)\rmd\tau +
\izt \Psi(\tau)\int_{0}^{s(\tau)}v_{tx}(x,\tau) \cdot v_{x}(x,\tau)\rmd x \rmd\tau
\]
\[
=-\izt \Psi(\tau)v_{t}(s(\tau),\tau)v_{x}(s(\tau),\tau)\rmd\tau+\frac{1}{2}\izt \Psi(\tau) \int_{0}^{s(\tau)} \frac{d}{d \tau}\abs{v_{x}(x,\tau)}^{2}\rmd x \rmd\tau.
\]
We note that for almost all $t \in (0,T]$ there holds
\[
0 = \frac{d}{dt}v(s(t),t) = v_{t}(s(t),t) + \dot{s}(t)v_{x}(s(t),t).
\]
Thus, we have
\[
-\izt \Psi(\tau) \int_{0}^{s(\tau)} v_{t}(x,\tau) \cdot v_{xx}(x,\tau)\rmd x\rmd\tau
=\izt\Psi(\tau)\dot{s}(\tau)\abs{v_{x}(s(\tau),\tau)}^{2}\rmd\tau
\]
\[
-\frac{1}{2}\izt \Psi(\tau)\dot{s}(\tau)\abs{v_{x}(s(\tau),\tau)}^{2}\rmd\tau
+\frac{1}{2}\izt \Psi(\tau) \frac{d}{d \tau}\int_{0}^{s(\tau)} \abs{v_{x}(x,\tau)}^{2}\rmd x \rmd\tau.
\]
Integrating the last integral by parts we obtain that
\[
-\izt \Psi(\tau) \int_{0}^{s(\tau)} v_{t}(x,\tau) \cdot v_{xx}(x,\tau)\rmd x\rmd\tau
=\frac{1}{2}\izt\Psi(\tau)\dot{s}(\tau)\abs{v_{x}(s(\tau),\tau)}^{2}\rmd\tau
\]
\eqq{
-\frac{1}{2}\izt \Psi'(\tau)\int_{0}^{s(\tau)} \abs{v_{x}(x,\tau)}^{2}\rmd x \rmd\tau
+\Psi(t)\int_{0}^{s(t)} \abs{v_{x}(x,t)}^{2}\rmd x.
}{noa}

To estimate the second term in \eqref{secest2} we apply Remark \ref{reAH}, estimate (\ref{sbind}) and Corollary \ref{eqc}. Then we get
\[
\int_{0}^{s(t)} \poch \da v(x,t)\cdot v_{xx}(x,t)\rmd x =
\int_{0}^{s(t)} \p^{\al}v_{x}(x,t)\cdot \p^{1-\al}\p^{\al}v_{x}(x,t)\rmd x
\]
\[
\geq
c \norm{ \p^{\al}v_{x}(\cdot,t)}^{2}_{H^{\frac{1-\al}{2}}(0,s(t))} \geq c \norm{ v_{x}(\cdot,t)}^{2}_{H^{\frac{1+\al}{2}}(0,s(t))},
\]
where $c = c(\al,M,T)$. Hence,
\eqq{
\izt \Psi(\tau) \int_{0}^{s(\tau)} \poch \da v(x,\tau)\cdot v_{xx}(x,\tau)\rmd x \rmd\tau
\geq \bar{c} \izt \Psi(\tau)  \norm{ v_{x}(\cdot,\tau)}^{2}_{H^{\frac{1+\al}{2}}(0,s(\tau))}\rmd\tau,
}{nob}
where $\bar{c} = c(\al,M,T)$ is fixed from now.
Let us now estimate the term on the right hand side of (\ref{secest2})
\[
\abs{\izt \Psi(\tau)\int_{0}^{s(\tau)} \bar{g}(x,\tau)\cdot v_{xx}(x,\tau)\rmd x \rmd\tau}
\leq \izt \Psi(\tau)\norm{\bar{g}(\cdot,\tau)}_{H^{\frac{1-\al}{2}}(0,s(\tau))} \norm{ v_{xx}(\cdot,\tau)}_{H^{\frac{\al-1}{2}}(0,s(\tau))}\rmd\tau
\]
\[
\leq \frac{\bar{c}}{2}\izt \Psi(\tau)  \norm{ v_{x}(\cdot,\tau)}^{2}_{H^{\frac{\al+1}{2}}(0,s(\tau))}\rmd\tau
+c(\bar{c}) \izt \Psi(\tau)\norm{\bar{g}(\cdot,\tau)}^{2}_{H^{\frac{1-\al}{2}}(0,s(\tau))} \rmd\tau.
\]
We note that, since $\Psi \equiv 0$ on $[0,t_{1}]$ and $\abs{\Psi} \leq 1$, we have
\eqq{
\izt \Psi(\tau)\norm{\bar{g}(\cdot,\tau)}^{2}_{H^{\frac{1-\al}{2}}(0,s(\tau))} \rmd\tau \leq  \int_{t_{1}}^{t}\norm{\bar{g}(\cdot,\tau)}^{2}_{H^{\frac{1-\al}{2}}(0,s(\tau))} \rmd\tau.
}{noc}
Recalling definition of $\bar{g}$  we obtain 
\eqq{
\int_{t_{1}}^{t}\norm{\bar{g}(\cdot,\tau)}^{2}_{H^{\frac{1-\al}{2}}(0,s(\tau))} \rmd\tau \leq \norm{h}_{W^{1,\infty}(0,T)}c(1/s(t_{1})).
}{nod}
Using (\ref{noa}), (\ref{nob}), (\ref{noc}) and (\ref{nod}) in (\ref{secest2}) we arrive at
\[
\frac{1}{2}\izt\Psi(\tau)\dot{s}(\tau)\abs{v_{x}(s(\tau),\tau)}^{2}\rmd\tau
+\Psi(t)\int_{0}^{s(t)} \abs{v_{x}(x,t)}^{2}\rmd x + \frac{\bar{c}}{2} \izt \Psi(\tau)  \norm{ v_{x}(\cdot,\tau)}^{2}_{H^{\frac{1+\al}{2}}(0,s(\tau))}\rmd \tau
\]
\eqq{
 \leq \frac{1}{2}\izt \Psi'(\tau)\int_{0}^{s(\tau)} \abs{v_{x}(x,\tau)}^{2}\rmd x \rmd\tau + \norm{h}_{W^{1,\infty}(0,T)}c(\bar{c},1/s(t_{1})).
}{nof}
 Let us estimate the first term on the right-hand-side. We note that for any fixed $l > 0$.
 \[
 H^{1}(0,l) = [L^{2}(0,l),H^{\frac{3+\al}{2}}(0,l)]_{\frac{2}{3+\al}}.
 \]
 Applying the interpolation inequality with $\theta =\frac{2}{3+\al}$ we obtain for any $\ve > 0$ that
 \[
 \norm{v(\cdot,t)}^{2}_{H^{1}(0,s(t))} \leq c(\theta)\ve  \norm{v(\cdot,t)}_{H^{\frac{3+\al}{2}}(0,s(t))}^{2\theta}\frac{1}{\ve}\norm{v(\cdot,t)}^{2(1-\theta)}_{L^{2}(0,s(t))}
 \]
 Making use of the Young's inequality with parameters $1/\theta$, $1/(1-\theta)$ we arrive at
 \[
 \norm{v(\cdot,t)}^{2}_{H^{1}(0,s(t))} \leq c(\theta)\left[\ve^{\frac{1}{\theta}}\norm{v(\cdot,t)}_{H^{\frac{3+\al}{2}}(0,s(t))}^{2} + \ve^{-\frac{1}{1-\theta}}\norm{v(\cdot,t)}^{2}_{L^{2}(0,s(t))}\right].
 \]
 We choose  $\ve^{1/\theta} = \frac{\bar{c}}{2 c(\theta)}\frac{\Psi(t)}{\Psi'(t)}$. Then, recalling that $\theta = \frac{2}{3+\al}$ we get that $\ve^{-\frac{1}{1-\theta}} = c(\al,M,T)\left(\frac{\Psi'(t)}{\Psi(t)}\right)^{\frac{2}{1+\al}}.$ Inserting this in (\ref{nof}) we obtain 
 \[
\frac{1}{2}\izt\Psi(\tau)\dot{s}(\tau)\abs{v_{x}(s(\tau),\tau)}^{2}\rmd\tau
+\Psi(t)\int_{0}^{s(t)} \abs{v_{x}(x,t)}^{2}dx + \frac{\bar{c}}{4} \izt \Psi(\tau)  \norm{ v_{x}(\cdot,\tau)}^{2}_{H^{\frac{1+\al}{2}}(0,s(\tau))}\rmd\tau
\]
\[
 \leq c(\al)\int_{t_{1}}^{2t_{1}}\Psi'(\tau)\left(\frac{\Psi'(\tau)}{\Psi(\tau)}\right)^{\frac{2}{1+\al}} \norm{v}^{2}_{L^{2}(0,s(\tau))}\rmd\tau + \norm{h}_{W^{1,\infty}(0,T)}c(\bar{c},1/s(t_{1})).
\]
We note that from (\ref{estfirst}) we obtain a bound for $\sup_{t}\norm{u(\cdot,t)}_{L^{2}(0,s(t))}$, which in view of identity (\ref{urep}), gives also that $\sup_{t}\norm{v(\cdot,t)}_{L^{2}(0,s(t))}$ is bounded  by a constant depending only on $\al,M,T,\norm{h}_{L^{2}(0,T)}$. Thus we may estimate
\[
\int_{t_{1}}^{2t_{1}}\Psi'(\tau)\left(\frac{\Psi'(\tau)}{\Psi(\tau)}\right)^{\frac{2}{1+\al}} \norm{v(\cdot,\tau)}^{2}_{L^{2}(0,s(\tau))}\rmd\tau \leq \sup_{t \in [0,T]}\norm{v(\cdot,t)}^{2}_{L^{2}(0,s(t))}\int_{t_{1}}^{2t_{1}}\Psi'(\tau)\left(\frac{\Psi'(\tau)}{\Psi(\tau)}\right)^{\frac{2}{1+\al}}\rmd\tau.
\]
We note that the last integral is finite if the function $\Psi$ is smooth enough. Indeed, it is enough to consider $\Psi$ which increases polynomially near $t_{1}$, where the degree of polynomial is greater than one.
Finally, making use of the fact that $\Psi \geq 0$ and $\Psi\equiv 1$ on $[2t_{1},T]$ we arrive at
 \[
\int_{2t_{1}}^{t}\dot{s}(\tau)\abs{v_{x}(s(\tau),\tau)}^{2}\rmd\tau
+\int_{0}^{s(t)} \abs{v_{x}(x,t)}^{2}\rmd x
\]
\[
+  \int_{2t_{1}}^{t} \norm{ v_{x}(\cdot,\tau)}^{2}_{H^{\frac{1+\al}{2}}(0,s(\tau))}\rmd\tau
 \leq c(\al,t_{1},M,T,\norm{h}_{W^{1,\infty}(0,T)}) \m{ for every } t > 2t_{1}.
\]
The last estimate finishes the proof.
\end{proof}

\subsection{Existence of a solution to FBP with $b=0$.}


\begin{theo}\label{bzerofinal} Let us assume that $h$ satisfies (\ref{zalh}). Then, the space-fractional Stefan problem \eqref{IFSSP-bis} admits a limit solution $(u,s)$  such that  $u \in C(Q_{s,T})$, $u \geq 0$ in $Q_{s,T}$, $s\in C^{0,1}[0,T]$, $0 \leq \dot{s} \leq \norm{h}_{L^{\infty}(0,T)}$ and
\eqq{
\norm{u(\cdot,t)}_{L^{2}(0,s(t))} \in L^{\infty}(0,T), \hd  \norm{u(\cdot,t)}_{H^{\frac{1+\al}{2}}(0,s(t))} \in L^{2}(0,T),
}{urj}
\eqq{
\norm{u_{t}(\cdot,t)}_{H^{\frac{1-\al}{2}}(0,s(t))},\hd  \norm{\poch \da u(\cdot,t)}_{H^{\frac{1-\al}{2}}(0,s(t))} \in L^{2}_{loc}(0,T].
}{urd}
\end{theo}
\begin{proof}

The proof of the theorem will be given by approximation.
Let us discuss the family of fractional Stefan problems \eqref{Stefan} with data $\left(\frac{1}{m},0,h \right)$ for every $m\in \bbN$. Let us denote by $(u^{m},s^{m})$ the family of solutions to these problems given by Theorem \ref{finala}. 
Note that, from Theorem \ref{finala} we have that $s^{m}$ are Lipschitz continuous and $0\leq \dot{s}^m(t)\leq M:=\norm{h}_{L^{\infty}(0,T)}$ for every $m\in \bbN$, which implies that 
\begin{equation}\label{bounds}
s^m(t)=\frac{1}{m}+\int_0^t\dot{s}^m(\tau)\rmd \tau\leq 1+MT
\end{equation}
and 
\begin{equation}\label{equicont-s}
|s^m(t)-s^m(\tau)|\leq M|t-\tau|. 
\end{equation}

Hence, we may deduce that the sequence $\left\{s^m\right\}_{m\in\bbN}$ is equicontinuous and uniformly bounded. Using the Ascoli-Arzelà theorem and the closedness of set $\Sigma$ defined in (\ref{Sigma}), we obtain that on the subsequence $s^{m}\rightarrow s$ uniformly on $[0,T]$, where $s \in C^{0,1}(0,T)$ is such that $0\leq \dot{s}(t)\leq M$, a.e. in $(0,T)$. 
Since, from Theorem \ref{desigfronteraN} we deduce that $\left\{s^m\right\}_{m\in\bbN}$ is a decreasing sequence, it holds that $s^m\geq s$ on $[0,T]$.
Our next aim is to investigate the convergence of the sequence $u^{m}$. We will prove that on the subsequence $u^{m}$ converges to $u$ which is a solution to the moving boundary problem  

\begin{equation}\label{Stefan0}
  \begin{array}{rll}
(i) & u_{t} - \poch \da u = 0 & \textrm{ in } Q_{s,T}, \\
(ii) & \da u(0,t) =h(t) & \textrm{ for  } t \in (0,T),\\
(iii) & u(t,s(t)) = 0 & \textrm{ for  } t \in (0,T), \\
\end{array}  \end{equation}
where $s$ is the limit function obtained above. We will obtain this result applying energy estimates and weak compactness argument. At first we recall that, for every $m\in \bbN$, the pair $(u^{m},s^{m})$ satisfies 
\begin{equation}\label{Stefanm}
  \begin{array}{rll}
(i) &  u^{m}_{t} - \poch \da u^{m} = 0 & \textrm{ in }  Q_{s,T}, \\
(ii) & \da u^{m}(0,t) =h(t), &\textrm{ for  } t \in (0,T)\\
(iii) & u^{m}(t,s^{m}(t)) = 0 & \textrm{ for  } t \in (0,T), \\
(iv) & u^{m}(x,0) =0 & \textrm{ for  } x \in (0,\frac{1}{m}). \\
\end{array} \end{equation}


In order to pass to the limit we extend $u_{sin}^{m}$ by zero on $\{(x,t):0\leq t<T, x \in (\sm(t),\sm(T))\}$ and we denote this extension by $U^{m}_{sin}$. We define two extensions of $v^{m}$. We denote by $V^m$ the extension by zero on $\{(x,t):0\leq t<T, x \in (\sm(t),\sm(T))\}$. Then we define $\bar{V}^m$ as an odd extension, i.e.
\[
\bar{V}^m(x,t) =
 \left\{ \begin{array}{ll}
v^{m}(x,t) & \textrm{ for }x \in [0,s^{m}(t)], \\
-v^{m}(2k\sm(t)-x,t) & \textrm{ for }x \in ((2k-1) \sm(t),2k s^{m}(t)] \cap [0,\sm(T)], \\
v^{m}(x-2k\sm(t),t) & \textrm{ for }x \in (2k \sm(t),(2k+1)s^{m}(t)]\cap [0,\sm(T)], \\
\end{array} \right.
\]
for $k = 1,2,\dots$
Let us denote $U^{m}:=V^{m} + U^{m}_{sin}$ and  $\bar{U}^{m}:=\bar{V}^{m} + U^{m}_{sin}$.
Then, from (\ref{estfirst}), we obtain that on the subsequence
\eqq{
U^{m} \rightharpoonup u \m{ in } L^{2}(0,T;H^{\frac{1+\al}{2}}(0,s(T))) \m{ and } U^{m} \overset{*}{\rightharpoonup} u \m{ in } L^{\infty}(0,T;L^{2}(0,s(T))).
}{zb1}
Let us fix $t_{1} > 0$. Then, for $t \geq t_{1}$ we have $s^m(t) \geq s(t) \geq s(t_{1})$. This together with (\ref{szac2}) leads to
\[
\sup_{t \in (t_{1},T)} \int_{0}^{s(T)}\abs{\bar{V}_{x}^m(x,t)}^{2}dx + \int_{t_{1}}^{t} \norm{\bar{V}_{x}^m(\cdot,t)}^{2}_{H^{\frac{1+\al}{2}}(0,s(T))}d\tau \leq c(t_{1},\al,M,T,\norm{h}_{W^{1,\infty}(0,T)}),
\]
hence
\[
\{\bar{V}^{m}\} \m{ is bounded in } L^{\infty}(t_{1},T;H^{1}(0,s(T))) \m{ and }
\{\bar{V}^{m}_{x} \} \m{ is bounded in } L^{2}(t_{1},T;{}_{0}H^{\frac{1+\al}{2}}(0,s(T))).
\]
From the last one we infer that
\[
\{\poch \da \bar{V}^{m}\} \m{ is bounded in } L^{2}(t_{1},T;H^{\frac{1-\al}{2}}(0,s(T))).
\]
We note that $U^{m}_{sin}$ converges uniformly to the function $U_{sin}$ given by
\[
U_{sin} =
 \left\{ \begin{array}{ll}
h(t)\frac{x^{\al}}{\Gamma(1+\al)} & \textrm{ in } \{(x,t):0<x\leq \frac{1}{2}s(t), 0<t<T\}, \\
\bar{\varphi} & \textrm{ in } \{(x,t):\frac{1}{2}s(t)<x\leq\frac{3}{4}s(t), 0<t<T\}, \\
0 &  \textrm{ in } \{(x,t):\frac{3}{4}s(t)<x\leq s(T), 0<t<T\}\\
\end{array} \right.
\]
where $\bar{\vf}$ is smooth function such that $ \bar{\vf}(\frac{1}{2}s(t),t) = h(t)\frac{(s(t))^{\al}}{2^{\al}\Gamma(1+\al)}$ and $ \bar{\vf}(\frac{3}{4}s(t),t) = 0$. Furthermore, by estimate (\ref{nod}) $\poch \da U^{m}_{sin}$ is bounded in $L^{2}(t_{1},T;H^{\frac{1-\al}{2}}(0,s(T)))$. Thus, we obtain
\eqq{
\{\poch \da \bar{U}^{m}\} \m{ is bounded in } L^{2}(t_{1},T;H^{\frac{1-\al}{2}}(0,s(T))).
}{zb3}
Since $u^{m}_{t} = \poch \da \um$ in $Q_{s,T}$ we obtain on the subsequence
\[
\bar{U}^{m}_{t}\rightharpoonup \chi  \m{ in } L^{2}(t_{1},T;H^{\frac{1-\al}{2}}(0,s(T))).
\]
Due to $\bar{U}^{m} = U^{m}$ on $Q_{s,T}$, choosing test functions in a definition of weak limit with a support in $Q_{s,T}$ and making use of (\ref{zb1}) we arrive at $\chi = u_{t}$ in a weak sense in $Q_{s,T}$.
On the other hand, from (\ref{zb1}) we obtain that $U^{m}_{x}\rightharpoonup u_{x}$ in $L^{2}(0,T;H^{\frac{\al-1}{2}}(0,s(T)))$ and by the continuity of fractional integral we obtain that $\da U^{m} = I^{1-\al}U^{m}_{x} \rightharpoonup \da u$ in $L^{2}(0,T;H^{\frac{1-\al}{2}}(0,s(T)))$.
Furthermore, using (\ref{zb3}) we get
\[
\poch \da \bar{U}^{m}\rightharpoonup \chi_{1} \m{ in } L^{2}(t_{1},T;H^{\frac{1-\al}{2}}(0,s(T)))
\]
on the subsequence. Again thanks to $\bar{U}^{m} = U^{m}$ on $Q_{s,T}$ we obtain that $\chi_{1} = \poch \da u$ in a weak sense in $Q_{s,T}$.
Now we will examine whether the limit function $u$ satisfies boundary conditions. At first we will show that one may choose a subsequence such that there exists $\chi_{2}$ such that
\eqq{
\bar{U}^{m}\rightarrow \chi_{2} \hd \m{uniformly on} \hd [t_{1},T]\times [0,s(T)].
}{unico}
We note that for every $\beta < \al+\frac{1}{2}$ we have $x^{\al} \in H^{\beta}(0,1)$. Thus  $U^{m}_{sin} \in C([0,T];H^{\beta}(0,s(T)))$ for every $\beta < \al+\frac{1}{2}$ and for every $t > t_{1}$ it converges to $U_{sin}$ in $H^{\beta}(0,s(T)))$ uniformly with respect to $t>t_{1}$.
Thus for every $\beta < \al+\frac{1}{2}$ we have that $\bar{U}^{m}$ is bounded in  $L^{\infty}(t_{1},T;H^{\beta}(0,s(T)))$. Since we have already shown that $\{\bar{U}^{m}_{t}\}$ is bounded in  $L^{2}(t_{1},T;H^{\frac{1-\al}{2}}(0,s(T)))$, it follows by the Aubin-Lions theorem that one may choose a subsequence such that $\bar{U}^{m} \rightarrow \chi_2$ in $C([t_{1},T];H^{\gamma}(0,s(T)))$ for every $\gamma \in (0,\al+\frac{1}{2})$. Thus by the Sobolev embedding we arrive at (\ref{unico}). Applying $\bar{U}^{m} = U^{m}$ on $Q_{s,T}$ by the uniqueness of the limit we obtain that $\chi_{2} \equiv u$ on $Q_{s,T}$. Using the fact that $\sm \rightarrow s$ uniformly on $[0,T]$ we arrive at 
\[
u(s(t),t) = 0 \m{ for every } t \in (t_{1},T).
\]
Finally, since  $\bar{V}^{m}_{t}$ is bounded in $L^{2}((t_{1},T);H^{\frac{1-\al}{2}}(0,s(T)))$ and $\bar{V}^{m}$ is bounded in \linebreak $L^{2}((t_{1},T);H^{\frac{3+\al}{2}}(0,s(T)))$, applying the Aubin-Lions theorem we obtain that there exists $v$ such that on the subsequence $\bar{V}^{m} \rightarrow v$ in $L^{2}((t_{1},T);H^{1+\al}(0,s(T)))$. The last one implies that $\da \bar{V}^{m}\rightarrow \da v$ in $L^{2}((t_{1},T);H^{1}(0,s(T)))$ and thus $(\da v)(0,t) = \da \bar{V}^{m}(0,t) = 0$ for a.a $t \in (t_{1},T)$, which leads to $\da u(0,t) = \da U_{sin}(0,t) = h(t)$ for a.a $t \in (t_{1},T)$.
Since $t_{1} > 0$ was chosen arbitrary, we obtain that there exists $u \in C(Q_{s,T})$, satisfying (\ref{urj}) and (\ref{urd}), such that
\[
u_{t} - \poch \da u = 0  \m{ in } H^{\frac{1-\al}{2}}(0,s(t)) \hd \m{a.e in} \hd Q_{s,T}, \hd \hd u(s(t),t) = 0 \m{ for every } t \in (0,T]
\]
and for a.a $t \in (0,T)$
\[
(\da u)(0,t) = h(t).
\]
Since for every $m \in \mathbb{N}$ we have $u^{m} \geq 0$, from (\ref{unico}) we infer that $u \geq 0$ on $Q_{s,T}$.
Finally, to prove that the Stefan condition is satisfied, we note that in view of Theorem \ref{condicionintegralNC} the Stefan condition may be equivalently written in the integral form
\begin{equation}\label{CondInt1Nm}
s^m(t)=\frac{1}{m}+\int_0^t h(\tau) \rmd \tau -\int_0^{s^m(t)}u^m(x,t)\rmd x.
\end{equation}
For every $t>0$ we consider the compact set ${t}\times [0,s^1(t)]$ and taking the limit when $m\to \infty$ and  using \eqref{unico} we deduce that 
\begin{equation}\label{CondInt1N0}
s(t)=\int_0^t h(\tau) \rmd \tau -\int_0^{s(t)}u(x,t)\rmd x,
\end{equation}
which proves that the fractional Stefan condition $(\ref{IFSSP}-v)$ holds in $(0,T)$. This finishes the proof of Theorem~\ref{bzerofinal}.
\end{proof}

\begin{note} Note that there may be solutions for a boundary condition $h$ such that  $\dot{h}$ is not bounded, as it was proved in \cite{RoTaVe:2020}, where the explicit solution to problem \eqref{IFSSP} for $h(t)=h_0 t^{-\frac{\al}{1+\al}}$ is  given by 
\begin{equation}\label{solFSP-N-new}
\begin{split}
v_\al(x,t)=\frac{h_0}{\Gamma(\alpha)} \int_{x/t^{1/(1+\al)}}^{\eta_\al} w^{\alpha-1}E_{\alpha,1+\frac{1}{\alpha},1}\left(-\frac{w^{1+\alpha}}{1+\alpha}\right)\rmd w\\
\end{split}
\end{equation}
\begin{equation}\label{freebP2}
s_\al(t)=\eta_\al t^{\frac{1}{1+\alpha}},\quad t\, \in \, (0,T).
\end{equation}
where $E_{\al,1+1/\al,1}$ is the three-parametric Mittag-Lefler function and  $\eta_\al\in \bbR^+$ is the unique solution to the equation
$$H_\al(x) = x, \quad x > 0,$$
and the function $H_\al$ is defined by \begin{equation}\label{fcH2}
H_\al(x)=h_0\left((1+\alpha)-\frac{1}{\Gamma(\alpha)}\int_0^x w\sigma_\al(w)\rmd w\right).
\end{equation}

\end{note}

\section{Conclusions}
We have discussed two kinds of phase-change problems for a one dimensional material. We have considered the model with a non-local heat flux and the fractional Neumann boundary condition at the fixed face $x=0$. We have proven the  existence and uniqueness of a regular solution in the case when the initial domain is occupied by liquid and solid. In the case when the initial domain is at the phase change temperature and it is not possible to transform the domain into a cylinder,  a limit  solution have been obtained by using an approximation method. 
In the proof of main results, many interesting properties for space-fractional Stefan problem with fractional Neumann boundary condition have been established, such as  the extremum principle, an integral condition equivalent to the fractional Stefan condition, a monotonicity result and energy estimates. \\
It is still an open problem the study of the regularity and uniqueness of the limit solution to problem~\eqref{IFSSP}.
 
\section{Acknowledgements}
\noindent The authors are grateful to Prof. Rico Zacher for the idea of reformulation the regularity assumptions in Proposition \ref{extremumPple} and Proposition \ref{nonpositivity}.\\
\noindent The first and third author were partly supported the Project PIP N$^\circ$ 0275 from CONICET--Universidad Austral. The first author was also partly supported by  the projects Austral N$^\circ$ 006-INV00020 (Rosario, Argentina) and European Unions Horizon 2020 research and innovation programme under the Marie Sklodowska-Curie Grant Agreement N$^\circ$ 823731 CONMECH. The second author was partly supported by National Sciences Center, Poland through 2017/26/M/ST1/00700 Grant.
\bibliographystyle{plain}

\bibliography{Roscani_BIBLIO_GENERAL_nombres_largos_2020}

\end{document}